\documentclass{article} 
\usepackage{latexsym,amsfonts,amsthm,amsmath,amscd,amssymb}
\usepackage[dvips]{graphicx}

\newtheorem{lemma}{Lemma}[section]
\newtheorem{thm}[lemma]{Theorem}

\newtheorem{prop}[lemma]{Proposition}
\newtheorem{cor}[lemma]{Corollary}

\newtheorem{example}[lemma]{Example}
\newtheorem{defn}[lemma]{Definition}

\newcommand\matR{{\mathbb{R}}}

\begin{document}

\title{Differential forms on diffeological spaces and diffeological gluing, I}

\author{Ekaterina~{\textsc Pervova}}

\maketitle

\begin{abstract}
\noindent This paper aims to describe the behavior of diffeological differential forms under the operation of gluing of diffeological spaces along a smooth map. In the diffeological context, two ways of looking at diffeological forms are available, that of the vector space $\Omega^m(X)$ of all $m$-forms, and that of the pseudo-bundle $\Lambda^m(X)$ of values of these forms. We describe the behavior of the former under a gluing of diffeological spaces.

\end{abstract}

\section*{Introduction}

The aim of this work is to examine the behavior of diffeological differential forms under the operation of diffeological gluing. One of the origins of the notion of a diffeological space \cite{So1}, \cite{So2} was the desire (or need) to provide a remedy to the fact that the category of smooth manifolds is not cartesian closed, whereas the category of diffeological spaces, while including the category of smooth manifolds, turns out to be so; it is furthermore complete and cocomplete. We should mention that  similar motivations led to the development of other similar categories (see \cite{St}, \cite{BH} for an overview), such as Chen spaces (\cite{chen1}---\cite{chen3}, see also \cite{BH}), Sikorski's differential spaces (\cite{sikorski}), Fr\"{o}licher spaces (\cite{frol}), and Smith spaces (\cite{smith}), see \cite{St} for a detailed discussion of interrelations between these structure, \cite{BH} for a comparison of Chen spaces and diffeological spaces, and \cite{watts} for illustration of differences between diffeological and differential spaces (this is particularly relevant in the context of the present work, as in \cite{sasin91}, \cite{sasin92} the author considered the behavior of differential forms on differential spaces obtaining results quite similar to ours; however, the work of Watts (\cite{watts}) shows that one set of those results is not readily deduced from the other).

Apart from such internal mathematical considerations, interest in endowing topological spaces that are not manifolds, with a smooth structure of a kind, arises in various applications. In \cite{optim1} --- \cite{optim3} the authors deemed that the context of diffeological spaces and the corresponding analogues of differential geometric concepts was more appropriate for their work in optimization than the usual context of smooth manifolds. The already-mentioned works by Sikorski and those by Sasin were geared towards obtaining a mathematical framework for studying space-time singularities. Finally, the theory of smooth polyhedral surfaces (\cite{pol-surf}) can, on one hand, be considered as a still another instance of an alternative notion of smooth structure, and on the other hand could in principle be applied to some biochemistry questions, such as for instance those of the size, the curvature, and two-sidedness of lipid rafts\footnote{Certain types of lipid aggregates in cellular membranes.} \cite{rafts}, this via the use of the well-established notion of the \emph{effective molecular shape} of a lipid \cite{rafts2}, which allows to represent lipid molecules by truncated cones, and accordingly the surface of a raft as a polyhedral surface, and the use of the polyhedral version of the Gauss-Bonnet theorem \cite{polya}, using so-called \emph{lipid annular shells} \cite{rafts2} to define principal curvatures in the neighborhood of each cholesterol\footnote{Cholesterol is the main component of lipid rafts, where its distribution is more or less regular and lattice-like.} molecule.\footnote{A number of preliminary considerations on such an approach had been done by Prof. Riccardo Zucchi and myself, with the main obstacle emerging being that the experimental data --- by other authors --- show that the annular shell of a cholesterol molecule consists in general of five other lipids, which fact does not allow for a ready definition of principal curvatures. On the other hand, theoretical calculations show that the number of \emph{disjoint cones} adjacent to a cholesterol-representing cone is four --- and this fact can in principle be used to define the principal curvatures.} It can also be observed that in principle (subject to the development of the necessary differential geometric tools in the diffeological context) this approach could be carried out in the context of diffeological spaces via endowing the truncated cones with a suitable diffeology, for instance, in the way it is done for corners \cite{GI-Z19}, \cite{I-Z24} and simplices \cite{hector}, \cite{CWtriang} (see \cite{II2015}, \cite{ntumba} for similar notions, employed also in \cite{KK2020}, \cite{KK2021}, and \cite{magnotTriang2} for yet another --- and particularly promising --- alternative version).

\paragraph{Diffeological setting} We state in the first section the precise definitions of the terms needed (more details can be found in the excellent book \cite{iglesiasBook} and the subsequent \cite{I-Z24}), but here is a rough description of our context. A \emph{diffeological space} is a set equipped with a \emph{diffeology}, a set of maps into it satisfying a few axioms, that are declared to be smooth. There are ensuing notions of smooth maps between such spaces, the induced diffeologies of all kinds, among which in particular the \emph{subset diffeology} and the \emph{quotient diffeology}, which provide for any subset, and any quotient, of a diffeological space, being in turn a diffeological space. 

The fact that all quotients inherit a natural diffeological structure, provides for the operation of diffeological gluing being well-defined in the diffeological context. As in the case of the similar notion in the case of differential spaces \cite{sasin91}, \cite{sasin92}, this operation is very much akin to the notion of topological gluing: given two sets $X_1$ and $X_2$ and a map $f:X_1\supset Y\to X_2$, the usual gluing procedure yields the space $(X_1\sqcup X_2)/_{x_2=f(x_1)}=:X_1\cup_fX_2$, which for a continuous $f$ between (subsets of) topological spaces has a natural topology. Now, the just-mentioned property of diffeology ensures the same thing, as long as we assume that $f$ is smooth as a map on $Y$, which inherits its diffeology from $X_1$.

A diffeological differential $m$-form on a diffeological space $X$ is a collection of usual $m$-forms, one for each plot, defined on the domain of the definition of the plot, and satisfying a very natural smooth compatibility condition. The collection of all $m$-forms on a given diffeological space is naturally a diffeological vector space, denoted by $\Omega^m(X)$.

\paragraph{Main results} Our main results regard the diffeological vector space $\Omega^m(X_1\cup_f X_2)$, for which we obtain the following statement.\vspace{5mm}

\noindent\textbf{Theorem 1}. \emph{Let $X_1$ and $X_2$ be two diffeological spaces, let $f:X_1\supset Y\to X_2$ be a smooth map, and let $i:Y\hookrightarrow X_1$ and $j:f(Y)\hookrightarrow X_2$ be the natural inclusions. Let $m\geqslant 0$ be an integer. Then $\Omega^m(X_1\cup_f X_2)$ is diffeomorphic to the subset of $\Omega^m(X_1)\oplus\Omega^m(X_2)$, denoted by $\Omega_f^m(X_1)\oplus_{comp}\Omega^m(X_2)$, consisting of all pairs $(\omega_1,\omega_2)$ such that $i^*\omega_1=f^*j^*\omega_2$ and $\omega_1$ is a so-called $f$-invariant form.}\vspace{5mm}

Two forms $\omega_1\in\Omega^m(X_1)$ and $\omega_2\in\Omega^m(X_2)$ satisfying $i^*\omega_1=f^*j^*\omega_2$ are said to be \textbf{\emph{compatible}}; the \textbf{$f$-invariance} of a form $\omega_1$ means essentially that $\omega_1$ induces a well-defined form on the space of orbits of $f$. It is quite evident that in general not every form makes part of a pair of compatible forms, nor is every form on $X$ $f$-invariant, so $\Omega_f^m(X_1)\oplus_{comp}\Omega^m(X_2)$ in general is not a subdirect product of $\Omega^m(X_1)$ and $\Omega^m(X_2)$. Notice also that the above theorem is quite analogous to Proposition 2.6 of \cite{sasin91}, although, as already mentioned, the two categories involved are different (this is particularly relevant to diffeological gluing, since one of the examples in \cite{watts} showing that the obvious functor between these categories is not a natural equivalence (Example 2.67), is in fact a space obtained by diffeological gluing).

After proving Theorem 1 (the proof being relatively straightforward) we then dedicate significant attention to the case of subsequent gluings of more than two spaces and to the inheritance (in a space obtained by gluing) of what we denote as $\mathcal{D}_{inv_1}^{r_1^*}(\cdot)=\mathcal{D}_{inv_2}^{r_2^*}(\cdot\cdot)$ conditions, these actually being a form of extendibility conditions for forms on the factors of gluing (Section 2.3).

\paragraph{Simple examples} For simple illustrations we will use throughout the paper the following three examples:
\begin{enumerate}
\item Let $X_1$, $X_2$ be two copies of $\mathbb{R}$ each endowed with the standard diffeology, let $Y\subset X_1$ be $[-1,1]$, and let $f:Y\to X_2$ be the identity on $[-1,1]$, $ $$f(x)=x$.
\item Let $X_1$, $X_2$ be two copies of $\mathbb{R}$ each endowed with the standard diffeology, let $Y=(0,+\infty)\subset X_1$, and let $f:Y\to X_2$ act by $f(x)=\frac{1}{x}$.
\item Let $X_1$, $X_2$ be two copies of $\mathbb{R}$ each endowed with the standard diffeology, let $Y=[-1,1]\subset X_1$, and let $f:Y\to X_2$ act by $f(x)=x^2$.
\end{enumerate}
The first of the corresponding spaces $X_1\cup_fX_2$ is a kind of one-step ladder, having the form of a capital H, the second is a line but with diffeology different from the standard one on the real line, and the third is a cross, but again diffeologically different from, say, the union of the coordinate axes in $\mathbb{R}^2$ with its usual subset diffeology.

In these examples we shall limit ourselves to considering 1-forms. Various other examples will be introduced as needed to illustrate specific points.

\paragraph{Corners} The $n$-dimensional corner is the subset $K^n=\{(x_1,\ldots,x_n)\,|\,x_i\geqslant 0\mbox{ for all }i=1,\ldots,n\}\subset\mathbb{R}^n$ endowed with the subset diffeology relative to its inclusion into the standard $\mathbb{R}^n$. The intersection of $K^n$ with any hyperplane $x_{i_1}=\ldots=x_{i_k}=0$ is called a \emph{stratum} (as the collection of all strata does indeed form a stratification of $K^n$, see \cite{GI-Z18}, \cite{GI-Z19}), although the term \emph{face} would also be natural to use. Differential forms on $K^n$ were studied in \cite{GI-Z19} (see also \cite{I-Z24}), where it was shown that they are actually restrictions of usual differential forms on the ambient $\mathbb{R}^n$. We provide the following example of gluing of corners, which is mainly used to illustrate the notions related to iterated gluings, in particular, the above-mentioned invariance and extendibility conditions:
\begin{enumerate}
\item[4.] Let $X_1$ be the $2$-dimensional corner, $X_1=\{(x_0,x_1)\in\mathbb{R}^2\,|\,x_0,x_1\geqslant 0\}$, and let $X_2,X_3$ be $1$-dimensional corners, $X_2,X_3=\{t\in\mathbb{R}\,|\,t\geqslant 0\}$. Let $Y$ be the subset of $X_1$ given by the condition $x_0x_1=0$, and let $f:Y\to X_2$ act by $f(x_0,x_1)=\left\{\begin{array}{ll} x_0 & \mbox{if }x_1=0\\ x_1 & \mbox{if }x_0=0 \end{array}\right.$. Let $Z_1$ be the subset of $X_1$ given by $x_1=0,\,0<x_0<\pi$, and let $g_1:Z_1\to X_3$ act by $g_1(x_0,0)=\sin x_0$. Finally, let $Z_2=(0,2)\subset X_2$, and $g_2:Z_2\to X_3$ acts by $g_2(t)=(t-1)^2$.
\end{enumerate}

Treatment of corners is in fact a preliminary for treatment of a more general case of simplicial complexes, which for reasons of space we postpone to a separate paper \cite{forms-smplx} (as we do for treatment of closely related pseudo-bundles \cite{formsII}).

\paragraph{Acknowledgments} The first version of this paper, which had laid the foundations of the present one, had greatly benefited from discussions with Prof. Riccardo Zucchi. I am also grateful to anonymous referees of various versions of this paper for a multitude of useful suggestions, which have all been implemented in the present version. I also would like to thank Patrick Iglesias-Zemmour for discussions regarding diffeological forms on manifolds with corners. Finally, I acknowledge
the MIUR Excellence Department Project awarded to the Department of Mathematics, University of
Pisa, CUP 157G22000700001.

\section{Main definitions}

We briefly recall here the basic definitions regarding diffeological spaces, diffeological vector spaces,  diffeological gluing, and diffeological differential forms.

\subsection{Diffeological spaces}

A \textbf{diffeological space} is  an arbitrary set $X$ equipped with a \textbf{diffeology}, which is a set $\mathcal{D}$ of maps $U\to X$, where $U$ is any domain in $\matR^n$, with varying $n$; the set $\mathcal{D}$ must possess the following properties: it must include all constant maps into $X$, for any $p\in\mathcal{D}$ its pre-composition $p\circ F$ with any usual smooth map $F$ must again belong to $\mathcal{D}$, and, if $p:U\to X$ is a set map and $U$ admits an open cover by some sub-domains $U_{\alpha}$ such that $p|_{U_{\alpha}}\in\mathcal{D}$, then necessarily $p\in\mathcal{D}$. The maps that compose a given diffeology $\mathcal{D}$ on $X$ are called \textbf{plots} of $\mathcal{D}$ (or of $X$). For two diffeologies $\mathcal{D}_1$ and $\mathcal{D}_2$ on the same set $X$, $\mathcal{D}_1$ is said to be \textbf{finer} than $\mathcal{D}_2$ if $\mathcal{D}_1\subseteq\mathcal{D}_2$; if this is the case, $\mathcal{D}_2$ is said to be \textbf{coarser} than $\mathcal{D}_1$.

Given two diffeological spaces $X_1$ and $X_2$, a set map $f:X_1\to X_2$ is said to be \textbf{smooth} if for any plot $p$ of $X_1$ the composition $f\circ p$ is a plot of $X_2$. A map $f$ is a \textbf{diffeomorphism} if it is smooth, bijective, and its inverse is smooth. It is said to be \textbf{subduction} if every plot of $X_2$ locally has form $f\circ p$ for some plot $p$ of $X_1$. More generally, the finest diffeology on $X_2$ for which $f$ is smooth is termed the \textbf{pushforward} of the diffeology of $X_1$ by $f$, and $f$ is a subduction if this coincides with the original diffeology on $X_2$.

Given one or more diffeological spaces, there are standard diffeological counterparts of all the basic set-theoretic and topological constructions, such as taking subspaces, quotients, direct products, and disjoint unions. In particular, any subset $X'$ of a diffeological space $X$ has the standard diffeology that is called the \textbf{subset diffeology} and that consists of precisely those plots of $X$ whose range is contained in $X'$; the quotient of $X$ by any equivalence relation $\sim$ has the \textbf{quotient diffeology} that is the pushforward of the diffeology of $X$ by the quotient projection $X\to X/\sim$. The direct product of a collection of diffeological spaces carries the \textbf{direct product diffeology} that is the coarsest diffeology such that all projections on individual factors are smooth; and the disjoint union has the \textbf{disjoint union diffeology} that is defined as the finest diffeology such that the inclusion of each component is a smooth map.

The space of all smooth maps from $X_1$ to $X_2$ is endowed with the \textbf{functional diffeology} defined as follows. A map $q:U\to C^{\infty}(X,Y)$ is a plot for the functional diffeology if and only if for every plot $p:U'\to X$ of $X$ the natural evaluation map $U\times U'\ni(u,u')\to q(u)(p(u'))\in Y$ is a plot of $Y$.

\subsection{Diffeological vector spaces}

A \textbf{diffeological vector space} is a set $V$ that is a vector space and a diffeological space at the same time, and the operations are smooth maps for the subset diffeology. All the basic operations on vector spaces (subspaces, quotients, direct sums, tensor products, and duals) have their diffeological counterparts (see \cite{iglesiasBook}, \cite{vincent}, \cite{wu}, as well as \cite{wu24} for some recent treatments of diffeological vector spaces), in the sense of there being a standard choice of diffeology on the resulting vector space. Thus, a subspace is endowed with the subset diffeology, the quotient space, with a quotient one, the direct sum carries the product diffeology, and the tensor product, the quotient diffeology relative to the product diffeology on the free product of its factors. The diffeological dual $V^*$ is defined as $C^{\infty}(V,\matR)$, where $\matR$ has standard diffeology, and the diffeology on $V^*$ is the functional diffeology. 

There is a related notion of a diffeological vector pseudo-bundle introduced in \cite{iglFibre} (appearing under different names in \cite{CWtangent}, \cite{pseudobundles}), which is employed in the aforementioned (and closely related to our present sibject) notion of the pseudo-bundle of values of differential forms introduced in \cite{iglesiasBook}.

\subsection{Diffeological gluing}

This concept is a natural counterpart of the usual topological gluing in the diffeological context. Let $X_1$ and $X_2$ be two diffeological spaces, and let $f:X_1\supset Y\to X_2$ be a smooth (for the subset diffeology on $Y$) map. The result of the diffeological gluing of $X_1$ to $X_2$ along $f$ is the space $X_1\cup_f X_2$ defined by $$X_1\cup_f X_2=(X_1\sqcup X_2)/\sim,$$ where $\sim$ is the equivalence relation determined by $f$, that is, $Y\ni y\sim f(y)$. The diffeology on $X_1\cup_f X_2$, called the \textbf{gluing diffeology}, is the pushforward of the disjoint union diffeology on $X_1\sqcup X_2$ by the quotient projection $\pi:X_1\sqcup X_2\to X_1\cup_f X_2$. Notice that the gluing diffeology is a rather weak diffeology, which frequently turns out to be finer than other natural diffeologies that the resulting space might carry (as it occurs for the union of the coordinate axes in $\matR^2$, whose gluing diffeology (relative to the gluing of the two standard axes at the origin) is finer than the subset diffeology relative to its inclusion in $\matR^2$, see again Example 2.67 in \cite{watts}). 

There is a technical convention, which comes in handy when working with glued spaces. It is based on the trivial observation that the maps $$\alpha_1:X_1\hookrightarrow X_1\sqcup X_2\to X_1\cup_f X_2\mbox{ and }\alpha_2:X_2\hookrightarrow X_1\sqcup X_2\to X_1\cup_f X_2,$$ where in both cases the second arrow is the quotient projection $\pi$, are subductions, and morever $\alpha_1|_{X_1\setminus Y}$ and $\alpha_2$ are inductions, with their ranges forming a disjoint cover of $X_1\cup_fX_2$.

Whenever it does not cause confusion, we shall use throughtout the paper the standard notation of $i:Y\hookrightarrow X_1$ and $j:f(Y)\hookrightarrow X_2$ for the standard inclusions of $Y$ and $f(Y)$ to $X_1$ and $X_2$ respectively.

\subsection{Diffeological differential $m$-forms}

For diffeological spaces, there exists a well-developed theory of differential $m$-forms on them (see \cite{iglesiasBook}, Chapter 6, for a detailed exposition). From here on and throughout the paper $m$ is a fixed nonnegative integer.

A \textbf{diffeological differential $m$-form} on a diffeological space $X$ is defined by assigning to each plot $p:\matR^n\supset U\to X$ a usual differential $m$-form $\omega(p)=\sum_{i_1<\ldots<i_m}f_{i_1\ldots i_m}dx^{i_1}\ldots dx^{i_m}\in\Lambda^m(U)$ such that this assignment satisfies the following compatibility condition: if $q:U'\to X$ is another plot of $X$ such that there exists a usual smooth map $F:U'\to U$ with $q=p\circ F$ then $\omega(q)=F^*\omega(p)$. If $\omega$ is a diffeological $m$-form on $X$, its \textbf{differential} $d\omega$ is defined by assigning to each plot $p$ of $X$ the form $d(\omega(p))$, which is clearly an $(m+1)$-form on $X$.

The set of all differential $m$-forms on $X$ is denoted by $\Omega^m(X)$. It is a diffeological vector space, with the addition and the scalar multiplication operations defined pointwise and with the functional diffeology defined as follows: a map $Q:U'\to\Omega^m(X)$ is a plot of $\Omega^mX)$ if and only if for every plot $p:\mathbb{R}^n\supseteq U\to X$ the map $U'\times U\to\Lambda^m(\matR^n)$ given by $(u',u)\mapsto Q(u')(p)(u)$ is smooth; the expression $Q(u')(p)$ stands for the $m$-form on $U$ that the differential $m$-form $Q(u')$ on $X$ assigns to the plot $p$.

\section{The space $\Omega^m(X_1\cup_f X_2)$}

Let $X_1$ and $X_2$ be two diffeological spaces, and let $f:X_1\supset Y\to X_2$ be a smooth map that defines a gluing between them. Since the space $X_1\cup_f X_2$ is a quotient of the disjoint union $X_1\sqcup X_2$, the natural projection $\pi:X_1\sqcup X_2\to X_1\cup_f X_2$ yields the corresponding pullback map $\pi^*:\Omega^m(X_1\cup_f X_2)\to\Omega^m(X_1\sqcup X_2)$; as we show immediately below, the latter space is diffeomorphic to $\Omega^m(X_1)\oplus\Omega^m(X_2)$. We then consider the image of $\pi^*$ (this space is sometimes called the space of \emph{basic} forms \cite{iglesiasBook}, \cite{karshon-watts}); we show that, although in general $\pi^*$ is not surjective,
it is a diffeomorphism with its image. Finally, we describe the structure of this image, and show that various specific conditions defining partial cases of interest are preserved under iterated gluings.

\subsection{The diffeomorphism $\Omega^m(X_1\sqcup X_2)\cong\Omega^m(X_1)\oplus\Omega^m(X_2)$}

This is a rather easy and, in any case, expected fact, but, since its proof does not seem to have been explicitly spelled out anywhere, we provide it.

\begin{thm}\label{omega:sqcup:thm}
The space $\Omega^m(X_1\sqcup X_2)$ is diffeomorphic, as a diffeological vector space, to the space $\Omega^m(X_1)\oplus\Omega^m(X_2)$.
\end{thm}

\begin{proof}
Let us first describe a bijection $\varphi:\Omega^m(X_1\sqcup X_2)\to\Omega^m(X_1)\oplus\Omega^m(X_2)$. Denote by $i_1:X_1\hookrightarrow X_1\sqcup X_2$, $i_2:X_2\hookrightarrow X_1\sqcup X_2$ the natural inclusions; since these are smooth maps, we can define, for every $\omega\in\Omega^m(X_1\sqcup X_2)$, the map $\varphi$ to be given by $\varphi(\omega)=(i_1^*\omega,i_2^*\omega)$.

The map $\varphi$ is obviously linear and smooth, and we need to show that it is bijective with a smooth inverse. Observe first that the maps $i_1,i_2$ are diffeomorphisms of their domains with their images. The map $\varphi^{-1}$ is given by assigning to any pair $(\omega_1,\omega_2)$, $\omega_j\in\Omega^m(X_j)$, $j=1,2$, the form $\omega_1\sqcup\omega_2$ defined as follows. Let $p:U\to X_1\sqcup X_2$ be a plot, let $U=U_1\sqcup U_2$, where each $U_j$, $j=1,2$, is such that $p(U_j)\subseteq i_j(X_j)$, and denote $p_j=p|_{U_j}$; notice that each $U_j$ is open and each $p_j$ is a plot. We now define $(\omega_1\sqcup\omega_2)(p_j)=\omega_j(i_j^{-1}\circ p_j)$ for any $j=1,2$ such that $U_j\neq\emptyset$. It is then clear that $\omega_1\sqcup\omega_2$ is a well-defined $m$-form on $X_1\sqcup X_2$, and $\varphi(\omega_1\sqcup\omega_2)=(\omega_1,\omega_2)$.

It remains to show that $\varphi^{-1}$ is smooth. Let $Q_j:U'\to\Omega^m(X_j)$ for $j=1,2$ be a pair of plots of $\Omega^m(X_1)$, $\Omega^m(X_2)$ respectively, and let us consider $\varphi^{-1}\circ(Q_1,Q_2):U'\to\Omega^m(X_1\sqcup X_2)$. By construction $(\varphi^{-1}\circ(Q_1,Q_2))(u')=Q_1(u')\sqcup Q_2(u')$; let $p:U\to X_1\sqcup X_2$ be a plot. As before, let $U=U_1\sqcup U_2$, where each $U_j$, $j=1,2$, is such that $p(U_j)\subseteq i_j(X_j)$, $p_j=p|_{U_i}$. Then for each $j$ we have $(Q_1(u')\sqcup Q_2(u'))(p_j)=Q_j(u')(i_j^{-1}\circ p_j)$, which is smooth in $u'$. Since both $U_j$ are open, this completes the proof.
\end{proof}

In view of this theorem we shall henceforth and consistently throughout the paper consider the pullback map $\pi^*$ as taking values in $\Omega^m(X_1)\oplus\Omega^m(X_2)$ (by abuse of notation, we will write $\pi^*$ while actually referring to $\varphi\circ\pi^*$).

\subsection{The space $\Omega^m(X_1\cup_fX_2)$ as the image of $\pi^*$}

We now turn to the image of $\pi^*$, describing its structure and showing that it is in fact a diffeomorphic image.

\paragraph{Compatible forms} In order to consider forms on the space $X_1\cup_fX_2$, we need a certain compatibility notion for pairs of forms $\omega_1\in\Omega^m(X_1)$, $\omega_2\in\Omega^m(X_2)$, which characterizes just when the pair $(\omega_1,\omega_2)$ belongs to the image of the pullback map $\pi^*$. The definition (already given in the Introduction) is as follows; recall the natural inclusions $i:Y\hookrightarrow X_1$ and $j:f(Y)\hookrightarrow X_2$, and the corresponding pullback maps $i^*:\Omega^m(X_1)\to\Omega^m(Y)$ and $j^*:\Omega^m(X_2)\to\Omega^m(f(Y))$.

\begin{defn} 
Let $\omega_1\in\Omega^m(X_1)$ and $\omega_2\in\Omega^m(X_2)$ be two forms. They are said to be \textbf{\emph{compatible}} (or, more precisely, \textbf{$f$-compatible}) if $i^*\omega_1=f^*j^*\omega_2$.
\end{defn}

\begin{example}\label{compty:ex}
As an overall observation, note that for any of the six spaces $X_i$ in Examples 1-3, a diffeological differential $1$-form $\omega_i$ on $X_i$ can in fact be identified with a usual differential $1$-form on $\mathbb{R}$ via considering its value $\omega_i(id_i)=h_idx_i$, where $h_i$ is an ordinary smooth function $\mathbb{R}\to\mathbb{R}$, on the identity plot $id_i:\mathbb{R}\to X_i$; its value on any other plot $H:U\to X_i$ is the usual pullback form $H^*\omega_i(id_i)$. In this sense we shall write $\omega_i=h_idx_i$ for a given $1$-form $\omega_i$ on $X_i$.

In the case of Example 4, we observe that, in general, it is known \cite{GI-Z19} that any diffeological differential $m$-form on an $n$-dimensional corner $K_n\subset\mathbb{R}^n$ is the restriction of an ordinary differential form on $\mathbb{R}^n$. It follows that every $m$-form on $K_n$ can be represented as $\sum_{0\leqslant i_1<\ldots<i_m\leqslant n-1}h_{i_1\ldots i_m}dx^{i_1}\ldots dx^{i_m}$ (that is, by an ordinary differential form), where all $h_{i_1\ldots i_m}$ are ordinary smooth functions $(0,+\infty)^n\to\mathbb{R}^n$ such that they and all their derivatives have continuous limits at all the points of $\partial K_n$, the boundary of $K_n$ (the equivalence of the two claims follows from \cite{seeley}). In particular, we shall use the latter presentation considering that all the $h_{i_1\ldots i_m}$ involved are naturally (by continuity) defined at all the points of the boundary of the corner.

Let us now consider the compatibility notion for these examples. For the Examples 1---3, let $\omega_1=h_1dx_1$ and $\omega_2=h_2dx_2$ be $1$-forms on $X_1$ and $X_2$ respectively; we will consider $h_1$ and $h_2$ to be defined on the same $\mathbb{R}$ when this does not create confusion. We then have:
\begin{enumerate}
\item In Example 1, the forms $\omega_1$ and $\omega_2$ are compatible if and only if $h_1|_{[-1,1]}=h_2|_{[-1,1]}$; in particular, every form on $\mathbb{R}^n$ has at least one (many more, actually) form compatible with it, for instance, we might take $h_1=h_2$ on the entire $\mathbb{R}$;
\item In Example 2, $\omega_1$ and $\omega_2$ are compatible if and only if $h_1(x_1)=-\frac{h_2(1/x_1)}{x_1^2}$ for all $x_1\in(0,+\infty)$; in particular, $h_2$ must be such that the righthand function in this equality, defined on $(0,+\infty)$, extend to the entire $\mathbb{R}$. Since many functions (for instance, the constant function corresponding to $\omega_2=dx_2$, do not possess this property, it is evident that there are forms on both $X_2$ and $X_1$ (such as the aforementioned $dx_2$ and, $f$ being the inverse of itself, $dx_1$), that do not possess any compatible counterpart. On the other hand, many do, for instance, the forms $\frac{dx_1}{1+x_1^2}$ and $-\frac{dx_2}{1+x_2^2}$ are compatible;
\item In Example 3, $\omega_1$ and $\omega_2$ are compatible if and only if $h_1(x_1)=2x_1h_2(x_1^2)$ for all $x_1\in[-1,1]$. It is thus evident that every form on $X_2$ has a compatible counterpart (for instance, $dx_2$ and $2x_1dx_1$ are compatible); on the other hand, the same is not true for forms on $X_1$, where, for instance, the form $(x_1+1)dx_1$ does not admit compatible counterparts: indeed, for a form $h_2dx_2$ to be compatible with it, we must have on $[-1,1]$ that $x_1+1=2x_1h_2(x_1^2)$ $\Leftrightarrow$ $h_2(x_1^2)=\frac12+\frac{1}{2x_1}$ for $x_1\neq 0$. However, $h_2(x_1^2)$ must be a smooth function on the entire $[-1,1]$, while $\frac12+\frac{1}{2x_1}$ has a nonremovable discontinuity at $0$, whence the claim;
\item In Example 4, we will write $h_i$ for a $0$-form on $X_i$, $i=1,2,3$, $h_{10}dx_0+h_{11}dx_1$ for a $1$-form on $X_1$, and $h_idt$, $i=2,3$, for a $1$-form on $X_2$ or $X_3$ (observe that, since the gluing sets are $1$-dimensional, every version of compatibility is automatic for $2$-forms on whichever space). We then have that $0$-forms $h_1,h_2$ are $f$-compatible iff $h_1(t,0)=h_1(0,t)=h_2(t)\,\,\forall\,t>0$, $1$-forms $h_{10}dx_0+h_{11}dx_1$, $h_2dt$ are $f$-compatible iff $h_{10}(t,0)=h_{11}(0,t)=h_2(t)\,\,\forall\,t>0$, $0$-forms $h_1,h_3$ are $g_1$-compatible iff $h_1(t,0)=h_3(\sin t)\,\,\forall t\in(0,\pi)$, $1$-forms $h_{10}dx_0+h_{11}dx_1$ and $h_3dt$ are $g_1$-compatible iff $h_{10}(t,0)=\cos t\cdot h_3(\sin t)\,\,\forall t\in(0,\pi)$, and  $0$-forms $h_2,h_3$ are $g_2$-compatible iff $h_2(t)=h_3((t-1)^2)\,\,\forall\,t\in(0,2)$, $1$-forms $h_2dt,h_3dt$ are $g_2$-compatible iff $h_2(t)=2(t-1)h_3((t-1)^2)\,\,\forall\,t\in(0,2)$.
\end{enumerate}
\end{example}

The space of all pairs $(\omega_1,\omega_2)$ such that $\omega_1$ and $\omega_2$ are compatible will be denoted by $\Omega^m(X_1)\oplus_{comp}\Omega^m(X_2)$. Compatible forms possess first of all the following property.

\begin{prop}\label{compatible:forms:in:omega:in:terms:of:pullbacks:prop}
Let $\omega_i\in\Omega^m(X_i)$ for $i=1,2$. Then $\omega_1$ and $\omega_2$ are compatible if and only if for all plots $p$ of $X_1$ whose range is contained in $Y$ we have $\omega_1(p)=\omega_2(f\circ p)$.
\end{prop}

\begin{proof}
Obviously, the plots $p$ as in the statement are in the one-to-one correspondence, via the inclusion map $i$, with plots of the subset diffeology on $Y$, that is, every such $p$ has a unique presentation as $p=i\circ p_1$ for a plot $p_1$ of $Y$. Since $i^*\omega_1(p_1)=\omega_1(p)$ and $f^*j^*\omega_2(p_1)=\omega_2(f\circ p)$, the equivalence in the statement is obvious.
\end{proof}

\paragraph{$f$-equivalent plots and $f$-invariant forms} The other needed property is that of $f$-invariance of a form on $X_1$, whose precise definition we now state. It is based on the auxiliary notion of \emph{$f$-equivalent plots}.

\begin{defn}
Two plots $p_1$ and $p_1'$ of $X_1$ are said to be \textbf{$f$-equivalent} if they have the same domain of definition $U$ and for all $u\in U$ such that $p_1(u)\neq p_1'(u)$ we have $p_1(u),p_1'(u)\in Y$ and $f(p_1(u))=f(p_1(u'))$.
\end{defn}

In other words, two plots on the same domain are $f$-equivalent if they differ only at points that are mapped by both plots to the domain of gluing, and among such, only at those points whose images are identified by $f$. In particular, $p_1,p_1'$ are $f$-equivalent if and only if they descend to the same plot of $X_1\cup_fX_2$, that is, if $\pi\circ i_1\circ p_1=\pi\circ i_1\circ p_1'$ (where $i_1:X_1\hookrightarrow X_1\sqcup X_2$ is as in the above section).

\begin{example}\label{f-equiv:ex}
Since in Examples 1 and 2 the gluing map is injective, every plot is $f$-equivalent only to itself. Whereas in Example 3 it is easy to find instances of $f$-equivalent distinct plots, for example, by letting $p_1,p_1':(-1,1)\to X_1$ act by $p_1(x)=x$, $p_1'(x)=-x$.

Let us now consider Example 4. We first make the following general observation: for any of our spaces $X_i$, $i=1,2,3$, and for any two plots $F_1,F_2:U\to X_i$ with the same domain of definition the set $\{u\in U\,|\,F_1(u)\neq F_2(u)\}$, which will be denoted by $U_{\neq}$, is open in $U$. Indeed, this set is the complement of the set of zeroes of $F_1-F_2$, which is closed as the set of zeroes of a usual smooth function with values in $\mathbb{R}^2$ or $\mathbb{R}$, whichever is the case (this is actually also true for Examples 1-3, but is less needed there). We now obtain that two plots $F_1,F_2:U\to X_1$ are $f$-equivalent if $F_1(U_{\neq})=F_2(U_{\neq})\subseteq\{(x_0,x_1)\,|\,x_0x_1=0\}$, and for all $u\in U_{\neq}$ we have $F_1(u)=\tau F_2(u)$, where $\tau:\mathbb{R}^2\to\mathbb{R}^2$ is the ``switch'' function, $\tau(x_0,x_1)=(x_1,x_0)$ (for instance, $F_1,F_2:(-1,1)\to X_1$, $F_1(t)=(t^2,0)$, $F_2(t)=(0,t^2)$ are $f$-equivalent), two plots  $F_1,F_2$ of $X_1$ with the same domain of definition $U$ are $g_1$-equivalent iff $F_1(U_{\neq})=F_2(U_{\neq})\subseteq(0,\pi)\times\{0\}$ and for all $u\in U_{\neq}$ we have $F_2^0(u)=\pi-F_1^0(u)$, and two plots $F_1,F_2$ of $X_2$ with the same domain of definition $U$ are $g_2$-equivalent iff $F_1(U_{\neq})=F_2(U_{\neq})\subseteq(0,2)$ and for all $u\in U_{\neq}$ we have the equality $F_2(u)=2-F_1(u)$.
\end{example}

\begin{defn}
A form $\omega_1\in\Omega^m(X_1)$ is said to be \textbf{$f$-invariant} if for any two $f$-equivalent plots $p_1,p_1':U\to X_1$ we have $\omega_1(p_1)=\omega_1(p_1')$.
\end{defn}

The space of all $f$-invariant $m$-forms will be denoted by $\Omega_f^m(X_1)$. It is easy to see that this is a vector subspace of $\Omega^m(X_1)$.

\begin{example}
In Examples 1 and 2, since every plot is $f$-equivalent only to itself, every form is automatically $f$-invariant. Let us consider Example 3. In this case the space $\Omega_f^1(X_1)$ of $f$-invariant forms consists precisely of forms $h_1(x_1)dx_1$ where $h_1:\mathbb{R}\to\mathbb{R}$ is a smooth function that is odd on $[-1,1]$. Indeed, let $\omega_1=h_1(x_1)dx_1$ be a $f$-invariant form, and let $p_1,p_1':(-1,1)\to X_1$ be plots of $X_1$ defined by $p_1(t)=t$, $p_1'(t)=-t$. Then $p_1,p_1'$ are $f$-equivalent, and we must have $\omega_1(p_1)=h_1(x_1)dx_1=\omega_1(p_1')=h_1(-x_1)d(-x_1)=-h_1(-x_1)dx_1$, that is, $h_1(x_1)=-h_1(-x_1)$, for all $x_1\in[-1,1]$.

Suppose now that $h_1$ is odd on $[-1,1]$. Observe that if $U$ is a domain in some $\mathbb{R}^n$ and $p,q:U\to\mathbb{R}$ are two smooth functions such that $(p(u))^2=(q(u))^2$ for all $u\in U$ and $p(u)\neq q(u)\Rightarrow |p(u)|\leqslant 1$, then $h(p(u))dp=h(q(u))dq$ on the entire $U$. Indeed, let $U_+=\{u\in U|p(u)=q(u)\}$ and let $U_-=\{u\in U|p(u)=-q(u)\}$; obviously, $p(U_-)$ and $q(U_-)$ are both contained in $[-1,1]$. Therefore the desired equality holds both on $\mbox{Int}(U_+)$ and on $\mbox{Int}(U_-)$. Since $U=U_+\cup U_-$ and $p,q$ are smooth, it holds on the entire $U$, whence the claim.

Let us now turn to Example 4. It follows from the above-given description of equivalent plots that a $0$-form $h_1$ on $X_1$ is $f$-invariant iff $h_1(t,0)=h_1(0,t)\,\,\forall\,t>0$, and a $1$-form $h_{10}dx_0+h_{11}dx_1$ is $f$-invariant iff $h_{10}(t,0)=h_{11}(0,t)\,\&\,h_{10}(0,t)=h_{11}(t,0)\,\,\forall\,t>0$; a $0$-form $h_1$ on $X_1$ is $g_1$-invariant iff $h_1(t,0)=h_1(\pi-t,0)\,\,\forall t\in(0,\pi)$, and a $1$-form $h_{10}dx_0+h_{11}dx_1$ is $g_1$-invariant iff $h_{10}(t,0)=-h_{10}(\pi-t,0)\,\,\forall t\in(0,\pi)$; and finally a $0$-form $h_2$ on $X_2$ is $g_2$-invariant iff $h_2(t)=h_2(2-t)\,\,\forall\,t\in(0,2)$ and a $1$-form $h_2dt$ is $g_2$-invariant iff $h_2(t)=-h_2(2-t)\,\,\forall\,t\in(0,2)$.
\end{example}

As Examples 3 and 4 suggest, in many cases the existence of a form compatible with a given $\omega_1$ implies the $f$-invariance of $\omega_1$ (even if $f$ is not injective). It is easy to see that this is the case for all $X_1,Y$ such that for every plot $p:U\to X_1$ we have $U\subset\,\,\mbox{Cl}(\mbox{Int}(p^{-1}(Y)))\,\cup\,\mbox{Cl}(\mbox{Int}(p^{-1}(X_1\setminus Y)))$. However, since a diffeology can in principle include, as a plot, any type of map, in general the inclusion $(i^*)^{-1}f^*j^*(\Omega^m(X_2))\subseteq\Omega_f^m(X_1)$ may not be guaranteed.

\paragraph{The induced $m$-form $\omega_1\cup_f\omega_2$ on $X_1\cup_f X_2$} Let us now describe a simple construction which, starting with an $f$-invariant $m$-form $\omega_1$ on $X_1$ and an $m$-form $\omega_2$ on $X_2$ compatible with $\omega_1$, yields an $m$-form $\omega_1\cup_f\omega_2$ on $X_1\cup_fX_2$. Let $p:U\to X_1\cup_f X_2$ be an arbitrary plot of $X_1\cup_f X_2$, let $u\in U$, and let $U'$ be a neighborhood of $u$ in $U$ such that $p|_{U'}$ lifts to a plot $p_j$ of $X_j$ for $j=1$ and/or $j=2$. We then set, for every $u'\in U'$, that $$(\omega_1\cup_f\omega_2)(p)(u')=\omega_j(p_j)(u'),$$ which presumably yields a usual differential form on $U'$, and, since $u$ is any, on the entire $U$. The purported form $\omega_1\cup_f\omega_2$ is defined by the assignment $$p\mapsto(\omega_1\cup_f\omega_2)(p).$$

\begin{lemma}
If $\omega_1\in\Omega_f^m(X_1)$ and $\omega_2\in\Omega^m(X_2)$ are compatible then $\omega_1\cup_f\omega_2$ is a well-defined diffeological differential $m$-form on $X_1\cup_fX_2$.
\end{lemma}

\begin{proof}
We need to show that for each plot $p:U\to X_1\cup_f X_2$ of $X_1\cup_f X_2$ the form $(\omega_1\cup_f\omega_2)(p)\in C^{\infty}(U,\Lambda^m(U))$ is well-defined. Suppose that there exist $u\in U$ and neighborhoods $U'$, $U''$ of $u$ such that $p|_{U'}$ lifts to a plot $p'$ of $X_i$, $p|_{U''}$ lifts to a plot $p''$ of $X_j$ ($i,j\in\{1,2\}$ and may or may not coincide), and $\bar{p}':=p'|_{U'\cap U''}\neq p''|_{U'\cap U''}=:\bar{p}''$.

Now, observe that, since $X_2$ smoothly injects into $X_1\cup_f X_2$, $\bar{p}'$ and $\bar{p}''$ cannot be both plots of $X_2$. Assume that one of them, say $\bar{p}'$, is a plot of $X_1$, while the other, $\bar{p}''$, is a plot of $X_2$. Since they project to the same map to $X_1\cup_f X_2$, we can conclude that $\bar{p}'$ takes values in $Y$ and $\bar{p}''=f\circ\bar{p}'$, so $$(\omega_1\cup_f\omega_2)(p|_{U'\cap U''})=\omega_1(\bar{p}')=\omega_2(f\circ\bar{p}')=\omega_2(\bar{p}''),$$ by the compatibility of the forms $\omega_1$ and $\omega_2$ with each other.

Assume now that $\bar{p}'$ and $\bar{p}''$ are both plots of $X_1$. Once again, since they project to the same plot of $X_1\cup_f X_2$, for every $u\in U'\cap U''$ such that $\bar{p}'(u)\neq\bar{p}''(u)$ we have $\bar{p}'(u),\bar{p}''(u)\in Y$ and $f(\bar{p}'(u))=f(\bar{p}''(u))$, that is, $\bar{p}',\bar{p}''$ are $f$-equivalent; since $\omega_1$ is by assumption $f$-invariant, we obtain that $$\omega_1(\bar{p}')=\omega_1(\bar{p}'')=(\omega_1\cup_f\omega_2)(p|{U'\cap U''}).$$ 

We thus conclude that for every plot $p:U\to X_1\cup_f X_2$ and for every point $u$ of $U$ there exists a neighborhood of $u$ in $U$ such that $(\omega_1\cup_f\omega_2)(p)$ is a uniquely defined differential form on that neighborhood. Therefore for every $p$ the form $(\omega_1\cup_f\omega_2)(p)\in C^{\infty}(U,\Lambda^m(U))$ is well-defined. It remains to show that the resulting collection $\{(\omega_1\cup_f\omega_2)(p)\}$ of the usual differential $m$-forms satisfies the smooth compatibility condition. Indeed, let $q:U'\to X_1\cup_f X_2$ be another plot of $X_1\cup_f X_2$ such that there exists a smooth map $F:U'\to U$ such that $q=p\circ F$, and let us show that $\omega(q)=F^*(\omega(p))$.

It is sufficient to consider the case where $U$ is small enough so that $p$ lifts either to a plot of $X_1$ or one of $X_2$. Suppose first that $p$ lifts to a plot $p_2$ of $X_2$. Then we have $p=\pi\circ p_2$, so $q=\pi\circ p_2\circ F$. Notice that $p_2\circ F$ is also a plot of $X_2$ and is a lift of $q$; thus, according to our definition $\omega(q)=\omega_2(p_2\circ F)=F^*(\omega_2(p_2))=F^*(\omega(p))$. If now $p$ lifts to a plot $p_1$ of $X_1$ the same argument is sufficient, whence the conclusion.
\end{proof}

Notice that by construction we can write (in the sense of the remark made after Theorem~\ref{omega:sqcup:thm}) $\pi^*(\omega_1\cup_f\omega_2)=(\omega_1,\omega_2)$.

\begin{example}
We can thus define a differential form on $X_1\cup_fX_2$ by choosing a pair of compatible forms $\omega_1,\omega_2$ with an $f$-invariant $\omega_1$ and indicating that the pair is meant to describe a form on $X_1\cup_fX_2$. For instance, $dx_1\cup_fdx_2$, $\frac{dx_1}{1+x_1^2}\cup_f\frac{-dx_2}{1+x_2^2}$, and $x_1dx_1\cup_f\frac12dx_2$ are forms on the spaces $X_1\cup_fX_2$ of Examples 1, 2, and 3 respectively. In Example 4, the following forms are $1$-forms on spaces $X_1\cup_fX_2$, $X_1\cup_{g_1}X_3$, $X_2\cup_{g_2}X_3$ respectively: $(x_0+x_1)(dx_0+dx_1)\cup_ftdt$, $\sin(x_0+x_1)dx_0\cup_{g_1}sds$, $2(t-1)^3dt\cup_{g_2}sds$.
\end{example}

\paragraph{The image of $\pi^*$ as a set} We now show that all forms on $X_1\cup_fX_2$ can be obtained as $\omega_1\cup_f\omega_2$ for some suitable $\omega_1,\omega_2$.

\begin{thm}\label{image:pullback:forms:thm}
Let $\omega_i$ be a differential $m$-form on $X_i$, for $i=1,2$. The pair $(\omega_1,\omega_2)$ belongs to the image of the pullback map $\pi^*$ if and only if $\omega_1$ and $\omega_2$ are compatible, and $\omega_1$ is $f$-invariant.
\end{thm}

\begin{proof}
The proof essentially consists in checking the requirements for a differential form on a given space $X$ to be a pullback of a form on a given quotient of $X$, that are given in Subsection 6.38 of \cite{iglesiasBook}. Suppose that $(\omega_1,\omega_2)=\pi^*\omega$ for some $\omega\in\Omega^m(X_1\cup_f X_2)$. To show that $\omega_1$ and $\omega_2$ are compatible, recall that by definition $\omega_i(p_i)=\omega(\pi\circ p_i)$ for $i=1,2$ and any plot $p_i$ of $X_i$; now, if $p_1$ is a plot for the subset diffeology of $Y$ then $f\circ p_1$ is a plot of $X_2$. Furthermore, $\pi\circ p_1=\pi\circ f\circ p_1$ by the very construction of $X_1\cup_f X_2$. Therefore we have $\omega_1(p_1)=\omega(\pi\circ p_1)=\omega(\pi\circ f\circ p_1)=\omega_2(f\circ p_1)$, as wanted. To show that $\omega_1$ is $f$-invariant, let $p_1,p_1'$ be $f$-equivalent plots. Then $\omega_1(p_1)=\omega(\pi\circ p_1)=\omega(\pi\circ p_1')=\omega_1(p_1')$, so $\omega_1$ is $f$-invariant.

To prove the reverse, it suffices to observe that, as has already been stated, given a compatible pair $\omega_1\in\Omega_f^m(X_1)$, $\omega_2\in\Omega^m(X_2)$, we have $(\omega_1,\omega_2)=\pi^*(\omega_1\cup_f\omega_2)$, whence the claim.
\end{proof}

\paragraph{The map $\pi^*$ is a diffeomorphism $\Omega^m(X_1\cup_f X_2)\cong\Omega_f^m(X_1)\oplus_{comp}\Omega^m(X_2)$} Consider the map $\mathcal{L}:\Omega_f^m(X_1)\oplus_{comp}\Omega^m(X_2)\to\Omega^m(X_1\cup_f X_2)$ given by the assignment $$(\omega_1,\omega_2)\mapsto\omega_1\cup_f\omega_2.$$ As has already been indicated, this is the inverse of the map $\pi^*$, and to prove that the latter is a diffeomorphism, it remains to establish the following claim.

\begin{thm}\label{omega:of:glued:inverse:smooth:thm}
The map $\mathcal{L}:\Omega_f^m(X_1)\oplus_{comp}\Omega^m(X_2)\to\Omega^m(X_1\cup_f X_2)$ is smooth.
\end{thm}

\begin{proof}
Consider a plot $Q:U\to\Omega_f^m(X_1)\oplus_{comp}\Omega^m(X_2)$. By definition of a subset diffeology and a product diffeology, we can assume that $U$ is small enough so that for every $u\in U$ we have $Q(u)=(Q_1(u),Q_2(u))$, where $Q_1:U\to\Omega_f^m(X_1)$ is a plot of $\Omega_f^m(X_1)$, $Q_2:U\to\Omega^m(X_2)$ is a plot of $\Omega^m(X_2)$, and for all $u\in U$ the forms $Q_1(u)$ and $Q_2(u)$ are compatible.

By definition of the standard diffeology on $\Omega^m(X_i)$, $Q_i$ being a plot of $\Omega^m(X_i)$ means that for every plot $p_i:U_i'\to X_i$ the map $U\times U_i'\to\Lambda^m(\matR^n)$ acting by $(u,u_i')\mapsto(Q_i(u))(p_i)(u_i')$, is smooth. The compatibility of the forms $Q_1(u)$ and $Q_2(u)$ means that $Q_2(u)(f\circ p_1)=Q_1(u)(p_1)$, for all $u$ and for all plots $p_1$ of the subset diffeology of $Y$.

Suppose we are given $Q_1$ and $Q_2$ satisfying these conditions, and let us show that $\mathcal{L}\circ(Q_1,Q_2):U\to\Omega^m(X_1\cup_f X_2)$ is a plot of $\Omega^m(X_1\cup_f X_2)$. Consider an arbitrary plot $p:U'\to X_1\cup_f X_2$ and the corresponding evaluation map $$(u,u')\mapsto(\mathcal{L}(Q_1(u),Q_2(u)))(p)(u'):U\times U'\to\Lambda^m(\matR^n),$$ and let us show that this map is smooth.

Assume, as is sufficient, that $U'$ is small enough so that $p$ lifts either to a plot $p_1$ of $X_1$, or a plot $p_2$ of $X_2$. It may furthermore lift to more than one plot of $X_1$, or it may lift to both a plot of $X_1$ and a plot of $X_2$. Suppose first that $p$ lifts to a precisely one plot, say a plot $p_i$ of $X_i$. Then $$(u,u')\mapsto(\mathcal{L}(Q_1(u),Q_2(u)))(p)(u')=Q_i(u)(p_i)(u')\in\Lambda^m(\matR^n);$$ this is a smooth map, since each $Q_i$ is a plot of $\Omega^m(X_i)$.

Suppose now that $p$ lifts to two distinct plots $p_1$ and $p_1'$ of $X_1$. In this case, however, $Q_1(u)(p_1)=Q_1(u)(p_1')$ because $Q_1(u)$ is $f$-invariant for all $u\in U$, so we get the desired conclusion as in the previous case. Finally, if $p$ lifts to both $p_1$ and $p_2$ (each $p_i$ being a plot of $X_i$) then $p_2=f\circ p_1$, and we obtain the claim by using the compatibility of the forms $Q_1(u),Q_2(u)$ for any given $u$.
\end{proof}

We thus obtain a complete proof of Theorem 1 stated in the Introduction.

\begin{cor}\label{omega:of:glued:cor}
The pullback map $\pi^*$ is a diffeomorphism $\Omega^m(X_1\cup_fX_2)\to\Omega_f^m(X_1)\oplus_{comp}\Omega^m(X_2)$.
\end{cor}

\begin{example}\label{spaces:one:glue:ex}
We now use the results just obtained to describe the space $\Omega^1(X_1\cup_fX_2)$ for Examples 1-3, as well as various spaces of forms for Example 4. Notice first that Corollary~\ref{omega:of:glued:cor} immediately implies that in all three cases the space $\Omega^1(X_1\cup_fX_2)$ is diffeomorphic, as a vector space, to a vector subspace of $C^\infty(\mathbb{R})\oplus C^\infty(\mathbb{R})$. Furthermore, we can establish the following claims.
\begin{enumerate} 
\item In the case of Example 1, the space $\Omega^1(X_1\cup_fX_2)$ is diffeomorphic to $C_{(-1,1)}^\infty(\mathbb{R})\oplus C^\infty(\mathbb{R})$, where $C_{(-1,1)}^\infty(\mathbb{R})$ is the space of all the smooth functions on $\mathbb{R}$ that vanish on $[-1,1]$ and $C^\infty(\mathbb{R})$ is the usual space of all smooth functions $\mathbb{R}\to\mathbb{R}$, considered both with the usual functional diffeology. The diffeomorphism $\Psi:\Omega^1(X_1\cup_fX_2)\to C_{(-1,1)}^\infty(\mathbb{R})\oplus C^\infty(\mathbb{R})$ acts by $h_1dx_1\cup_fh_2dx_2\mapsto(h_1-h_2,h_2)$; this is obviously smooth and linear, with a smooth linear inverse. 
\item In the case of Example 2, the space $\Omega^1(X_1\cup_fX_2)$ is isomorphic to $C_{(0,+\infty)}^\infty(\mathbb{R})\oplus C_{1/*}^\infty(\mathbb{R})$, where $C_{(0,+\infty)}^\infty(\mathbb{R})$ is the space of those smooth functions on $\mathbb{R}$ that vanish on $(0,+\infty)$ and $C_{1/*}^\infty(\mathbb{R})$ is the space of the smooth functions $h_2$ on $\mathbb{R}$ such that the limits at $0$ of all the derivatives of $\frac{h_2(1/x)}{x^2}$ exist finite. We define the isomorphism $\Psi$ as follows. Fix a datum $(\{a_k\},\{b_k\},\phi_k)$ for the extension operator $E$ constructed by Seeley in \cite{seeley} (recall that $E$ assigns to every smooth function $\mathbb{R}^n\times\mathbb{R}_+\to\mathbb{R}$ that has all the continuous derivatives as $\mathbb{R}_+\ni x\to 0$, its smooth extension to the entire $\mathbb{R}^{n+1}$, for whatever integer $n\geqslant 0$). Let $\omega=h_1dx_1\cup_fh_2dx_2$ be a form in $\Omega^1(X_1\cup_fX_2)$. An obvious isomorphism $\Psi:\Omega^1(X_1\cup_fX_2)\to C_{(0,+\infty)}^\infty(\mathbb{R})\oplus C_{1/*}^\infty(\mathbb{R})$ is given by setting $\Psi(\omega)=(h_1(x_1)+E(\frac{h_2(1/x_1)}{x_1^2}),h_2)$. Indeed, its linearity is guaranteed by Seeley's theorem, and it has an obvious linear inverse. Also, it is easy to see that $\Omega^1(X_1\cup_fX_2)$ fibers over $C_{1/*}^\infty(\mathbb{R})$, via the map $h_1dx_1\cup_fh_2dx_2\mapsto h_2$, with the fiber diffeomorphic to $C_{(0,+\infty)}^\infty(\mathbb{R})$; note that if we can ensure that $E$ is a (diffeological) diffeomorphism, then $\Psi$ is a diffeomorphism as well.
\item In the case of Example 3, the space $\Omega^1(X_1\cup_fX_2)$ is again (linearly) diffeomorphic to $C_{(-1,1)}^\infty(\mathbb{R})\oplus C^\infty(\mathbb{R})$. Indeed, we define the desired diffeomorphism $\Psi:\Omega^1(X_1\cup_fX_2)\to C_{(-1,1)}^\infty(\mathbb{R})\oplus C^\infty(\mathbb{R})$ by setting, for $\omega=h_1dx_1\cup_fh_2dx_2$, $\Psi(\omega)=(h_1(x_1)-2x_1h_2(x_1^2),h_2)$, which is obviously smooth, linear, and has a smooth linear inverse. 
\item In Example 4, let us first introduce the following notation. For sets $X$, where $X$ is either $[0,+\infty)$ or $[0,+\infty)^2$ we denote by $C^{\infty}(X)$ the set of all usual smooth functions on the interior of $X$ that can be extended to a smooth function on the ambient Euclidean space; for a subset $A$ of the boundary of $X$, we denote $C_A^{\infty}(X)$ the subset of $C^{\infty}(X)$ consisting of all those functions that are flat at all points of $A$. We then easily obtain that: $\Omega^0(X_1\cup_fX_2)\cong C_{x_0x_1=0}^{\infty}([0,+\infty)^2)\times C^{\infty}([0,+\infty))$ via the diffeomorphism acting by $h_1\cup_fh_2\mapsto(h_1-(h_2,h_2),h_2)$, $\Omega^1(X_1\cup_fX_2)\cong C_{x_1=0}^{\infty}([0,+\infty)^2)\times C_{x_0x_1=0}^{\infty}([0,+\infty)^2)\times C^{\infty}([0,+\infty))$ via the diffeomorphism acting by $(h_0dx_0+h_1dx_1)\cup_fh_2dt\mapsto(h_0-(h_2,0),h_1-h_0^0-h_0^1,h_2)$, where $h_0^0(x_0,x_1)=h_0(x_0,0)$, $h_0^1(x_0,x_1)=h_0(0,x_1))$; $\Omega^0(X_1\cup_{g_1}X_3)\cong C_{(0,\pi)\times\{0\}}^{\infty}([0,+\infty)^2)\times C^{\infty}([0,+\infty))\times C^{\infty}([0,+\infty))$ via the diffeomorphism acting by $h_1\cup_{g_1}h_3\mapsto(h_1-(h_3\circ\sin,0),h_3)$, $\Omega^1(X_1\cup_{g_1}X_3)\cong C_{(0,\pi)\times\{0\}}^{\infty}([0,+\infty)^2)\times C^{\infty}([0,+\infty))$ via the diffeomorphism acting by $(h_0dx_0+h_1dx_1)\cup_{g_1}h_3dt\mapsto(h_1-(\cos\cdot(h_3\circ\sin),0),h_1,h_3)$; and $\Omega^0(X_2\cup_{g_2}X_3)\cong C_{(0,2)}^{\infty}([0,+\infty))\times C^{\infty}([0,+\infty))$ via the diffeomorphism acting by $h_2\cup_{g_2}h_3\mapsto(h_2(t)-h_3((t-1)^2),h_3)$, $\Omega^1(X_2\cup_{g_2}X_3)\cong C_{(0,2)}^{\infty}([0,+\infty))\times C^{\infty}([0,+\infty))$ via the diffeomorphism acting by $h_2dt\cup_{g_2}h_3ds\mapsto(h_2(t)-2(t-1)h_3((t-1)^2),h_3)$.
\end{enumerate}
\end{example}

\paragraph{The Souriau's de Rham complex $\Omega_{dR}(X_1\cup_fX_2)$} Theorem 1 allows to obtain immediately the following simple observation regarding the de Rham complex $\Omega_{dR}(X_1\cup_fX_2)$. Recall first that there are two notions of the de Rham complex in diffeology. In this section we briefly consider the original one, due to Souriau \cite{So1}; do note that there is an alternative notion, the \emph{singular de Rham complex}, which is not equivalent to the Souriau's one and is due to Kuribayashi \cite{KK2020} (see \cite{KK2021} for a comparison of the two).

Let $\Omega_{dR}(X)$ stand for the Souriau's de Rham complex of a given diffeological space $X$. The following easy observations allow us to comment on $\Omega_{dR}(X_1\cup_fX_2)$.

\begin{lemma}
Let $\omega_i\in\Omega^m(X_i)$, $i=1,2$. If $\omega_1$ and $\omega_2$ are compatible then so are $d\omega_1$ and $d\omega_2$.
\end{lemma}

\begin{proof}
Let us show that $i^*d\omega_1=f^*j^*d\omega_2$. Let $p:U\to Y$ be a plot of $Y$. Then $i^*d\omega_1(p)=d\omega_1(i\circ p)=d\omega_2(j\circ f\circ p)$ by the assumption on $\omega_1$ and $\omega_2$, and we obtain the desired equality.
\end{proof}

We furthermore have the following property.

\begin{lemma}
If $\omega_1\in\Omega_f^m(X_1)$ and $\omega_2\in\Omega^m(X_2)$ are compatible then $$d(\omega_1\cup_f\omega_2)=d\omega_1\cup_fd\omega_2.$$
\end{lemma}

\begin{proof}
Let $p:U\to X_1\cup_fX_2$ be a plot. It is sufficient to consider the case when $U$ is small enough so that $p$ lifts to either a plot $p_1$ of $X_1$ or to a plot $p_2$ of $X_2$. Suppose $p$ lifts to $p_1$; then $d(\omega_1\cup_f\omega_2)(p)=d\left((\omega_1\cup_f\omega_2)(p)\right)=d\omega_1(p_1)$ by definition of the form $\omega_1\cup_f\omega_2$, while $(d\omega_1\cup_fd\omega_2)(p)=d\omega_1(p_1)$ by the same definition applied to the forms $d\omega_1$ and $d\omega_2$. Likewise, if $p$ lifts to $p_2$, we have $d(\omega_1\cup_f\omega_2)(p)=d\omega_2(p_2)=(d\omega_1\cup_fd\omega_2)(p)$, whence the claim.
\end{proof}

The two lemmas, together with Corollary~\ref{omega:of:glued:cor}, immediately imply the following statement.

\begin{thm}
The complex $\Omega_{dR}(X_1\cup_fX_2)$ is a, in general proper, subcomplex of the direct sum complex $\Omega_{dR}(X_1)\oplus\Omega_{dR}(X_2)$.
\end{thm}

Of course, it also follows that the de Rham groups $H_{dR}^m(X_1\cup_fX_2)$ are subgroups of the direct sums $H_{dR}^m(X_1)\oplus H_{dR}^m(X_2)$.

\paragraph{The projections of $\Omega_f^m(X_1)\oplus_{comp}\Omega^m(X_2)$ to $\Omega^m(X_1)$ and to $\Omega^m(X_2)$} We now provide some extra considerations on the space $\Omega_f^m(X_1)\oplus_{comp}\Omega^m(X_2)$, particularly in terms of the potential surjectivity, or lack of same, of its projections on the two factors of the direct sum (which question corresponds of course to that of it being possible to extend a given form on $X_i$ to a form on $X_1\cup_fX_2$). To begin with, note that particularly for a non-injective $f$, it is quite clear that at least the former of these two projections cannot be expected to always be surjective: indeed, the projection on $\Omega^m(X_1)$ contains $f$-invariant forms only, which in general is a space smaller than the entire $\Omega^m(X_1)$ (see Example 3). Furthermore, not every $f$-invariant form is in the image of the projection on $\Omega^m(X_1)$; for instance, in Example 2 the form $dx_1$ is $f$-invariant but it is not in the image of the projection $\Omega_f^1(X_1)\oplus_{comp}\Omega^1(X_2)\to\Omega^1(X_1)$, since $(f^{-1})^*(dx_1)=-\frac{dx_2}{x_2^2}$ does not extend to the entire $X_2$ (in other words, $f$-invariance of a form $\omega_1$ does not by itself imply the existence of a form $\omega_2$ compatible with it). Furthermore, both Examples 2 and 3 show that compatibility is also a restrictive condition, in terms of when the two projections could be surjective.

An evident criterion for the two projections being surjective can however be stated in terms of the pullback maps $i^*:\Omega^m(X_1)\to\Omega^m(Y)$ and $j^*:\Omega^m(X_2)\to\Omega^m(f(Y))$.

\begin{prop}\label{when:prod:comp:is:sub-direct:prop}
The two projections $\mbox{pr}_1:\Omega_f^m(X_1)\oplus_{comp}\Omega^m(X_2)\to\Omega^m(X_1)$ and $\mbox{pr}_2:\Omega_f^m(X_1)\oplus_{comp}\Omega^m(X_2)\to\Omega^m(X_2)$ are both surjective if and only if every form on $X_1$ is $f$-invariant, and the following is true: $$i^*(\Omega^m(X_1))=(f^*j^*)(\Omega^m(X_2)).$$
\end{prop}

\begin{example}
In Example 1 the spaces $i^*(\Omega^1(X_1))$ and $(f^*j^*)(\Omega^1(X_2))$ are both equal to the set of forms of type $h_1dx_1$, where $h_1:(-1,1)\to\mathbb{R}$ is any smooth function that admits an extension to the entire $\mathbb{R}$; in particular the two spaces coincide, so the two projections are surjective. In Example 2, the space $i^*(\Omega^1(X_1))$ is the space of all forms of type $h_1dx_1$, where $h_1:(0,+\infty)\to\mathbb{R}$ is any smooth function that admits an extension to the entire $\mathbb{R}$. On the other hand, the space $(f^*j^*)(\Omega^1(X_2))$ is the space of forms of shape $-\frac{h_2(1/x_1)}{x_1^2}dx_1$, where $h_2$ is any smooth function on $\mathbb{R}$; the latter space is obviously different from $i^*(\Omega^1(X_1))$, since it contains forms (for instance, $-\frac{dx_1}{x_1^2}$) that are not defined on the entire $\mathbb{R}$; and in fact neither of the two projections is surjective. Finally, in Example 3 the space $i^*(\Omega^1(X_1))$ is again equal to the set of forms of type $h_1dx_1$, where $h_1:(-1,1)\to\mathbb{R}$ is any smooth function that admits an extension to the entire $\mathbb{R}$, whereas the space $(f^*j^*)(\Omega^1(X_2))$ is the space of forms of shape $2x_1h_2(x_1^2)dx_1$, where $h_2:\mathbb{R}\to\mathbb{R}$ is any smooth function. This space is actually smaller than $i^*(\Omega^1(X_1))$. Indeed, consider the form $(x_1+1)dx_1\in i^*(\Omega^1(X_1))$; if we were to have, for some $h_2\in C^\infty(\mathbb{R})$, that $(x_1+1)dx_1=2x_1h_2(x_1^2)dx_1$, on $(-1,1)\setminus\{0\}$ we would have $\frac12+\frac{1}{2x_1}=h_2(x_1^2)$, \emph{i.e.} that $\frac12+\frac{1}{2x_1}$ would be an even function on $(-1,1)\setminus\{0\}$, which is obviously not the case. Thus, of the two projections, that on the second factor is surjective, while the other is not so. Example 4 will be considered below (see Example~\ref{phfrd:dfgies:expl}).
\end{example}

\paragraph{The two pushforward diffeologies $\mathcal{D}_f^{i^*}(Y)$ and $\mathcal{D}^{f^*j^*}(Y)$}\label{two:diffeologies:condition:sect} A further version of the above extendibility criterion is as follows; it corresponds to a possibility to extend a given \emph{plot} of $\Omega^m(X_i)$ to a plot of $\Omega^m(X_1\cup_fX_2)$.

Denote by $\mathcal{D}_f^{i^*}(Y)$ the diffeology on $\Omega^m(Y)$ that is the pushforward of the diffeology of $\Omega_f^m(X_1)$ by the map $i^*$. Since this map is smooth for the standard functional diffeology on $\Omega^m(Y)$, $\mathcal{D}_f^{i^*}(Y)$ is contained in this standard diffeology; notice that it may be properly contained. Next, let $\mathcal{D}^{f^*j^*}(Y)$ be another diffeology on $\Omega^m(Y)$, and precisely the one obtained as the pushforward of the standard functional diffeology on $\Omega^m(X_2)$ by the map $f^*j^*$. Also in this case, it is contained in the standard diffeology of $\Omega^m(Y)$, perhaps properly. The further extendibility condition is that $$\mathcal{D}_f^{i^*}(Y)=\mathcal{D}^{f^*j^*}(Y).$$ It is satisfied in many standard contexts, such as those of connected simply-connected domains in $\matR^n$, or its affine subsets (see also Example~\ref{phfrd:dfgies:expl} below). Moreover, we observe that this assumption implies the equality  $i^*(\Omega_f^m(X_1))=(f^*j^*)(\Omega^m(X_2))$.

\begin{prop}
Let $X_1$ and $X_2$ be two diffeological spaces, and let $f:X_1\supseteq Y\to X_2$ be a gluing map. If $\mathcal{D}_f^{i^*}(Y)=\mathcal{D}^{f^*j^*}(Y)$ then $i^*(\Omega_f^m(X_1))=(f^*j^*)(\Omega^m(X_2))$.
\end{prop}

\begin{proof}
Let $i^*(\omega_1)\in i^*(\Omega_f^m(X_1))$, where $\omega_1\in\Omega_f^m(X_1)$ is any $f$-invariant $m$-form. Consider a constant map $p_1^{\Omega}:U\to\{\omega_1\}\subset\Omega_f^m(X_1)$; this is a plot of $\Omega_f^m(X_1)$ since all constant maps are so. By assumption, $i^*\circ p_1^{\Omega}\in\mathcal{D}_f^{i^*}(Y)=\mathcal{D}^{f^*j^*}(Y)$. Since $\mathcal{D}^{f^*j^*}(Y)$ is defined as the pushforward of the diffeology of $\Omega^m(X_2)$ by the map $f^*j^*$, there exists a plot $p_2^{\Omega}:U\to\Omega^m(X_2)$ of $\Omega^m(X_2)$ such that $i^*\circ p_1^{\Omega}=(f^*j^*)\circ p_2^{\Omega}$. Let $\omega_2\in\Omega^m(X_2)$ be any form in the range of $p_2^{\Omega}$. Then $i^*\omega_1=f^*j^*\omega_2$, which means that $i^*(\Omega^m(X_1))\subseteq(f^*j^*)(\Omega^m(X_2))$. The reverse inclusion is proved in exactly the same way, therefore the claim.
\end{proof}

\subsection{Iterated gluings}

We now consider finite sequences of subsequent gluings.

\subsubsection{Notation}\label{notation:iter:sect} 

When it comes to iterated gluings, the notation starts getting cumbersome, so we dedicate a separate section to spelling it out. Let $X_1,X_2,X_3$ be diffeological spaces, let $Y\subseteq X_1$, let $Z_1\subseteq X_1$, $Z_2\subseteq X_2$, let $f:Y\to X_2$ be a subduction, and let $g_1:Z_1\to X_3$, $g_2:Z_2\to X_3$ be smooth maps, where $g_1$ is such that $f(y)=f(y')\Rightarrow g_1(y)=g_1(y')$ for all $y,y'\in Y\cap Z_1$; let also $$i:Y\hookrightarrow X_1,\,\,\,j:f(Y)\hookrightarrow X_2,\,\,\,k_i:Z_i\hookrightarrow X_i,\,\,\,l_i:g_i(Z_i)\hookrightarrow X_3,\,\,\,i=1,2$$ be the natural inclusions. Recall that $\alpha_1:X_1\hookrightarrow X_1\sqcup X_2\xrightarrow{\pi}X_1\cup_fX_2$ and $\alpha_2:X_2\hookrightarrow X_1\sqcup X_2\xrightarrow{\pi}X_1\cup_fX_2$; denote by $\bar{k}_1$ and $\bar{k}_2$ the natural inclusions of, respectively, $\alpha_1(Z_1)$ and $\alpha_2(Z_2)$ in $X_1\cup_fX_2$; note that $\alpha_1\circ k_1=\bar{k}_1\circ\alpha_1|_{Z_1}$ and $\bar{k}_2\circ \alpha_2|_{Z_2}=\alpha_2\circ k_2$.

We set $\tilde{g}_1:\alpha_1(Z_1)\to X_3$ to be such that $g_1=\tilde{g}_1\circ\alpha_1|_{Z_1}$, and $\tilde{g}_2:\alpha_2(Z_2)\rightarrow X_3$ to be such that $g_2=\tilde{g}_2\circ \alpha_2|_{Z_2}$ (observe that $\tilde{g}_2$ is always defined, whereas defining $\tilde{g}_1$ requires the above-stated assumption on $g_1$), and consider the spaces $(X_1\cup_fX_2)\cup_{\tilde{g}_1}X_3$ and $(X_1\cup_fX_2)\cup_{\tilde{g}_2}X_3$, determining in Sections~\ref{invarnce:sect}---\ref{case:diffeo:sect} the corresponding spaces of differential $m$-forms. Observe that these two spaces are not the only possible way to interpret what it means to glue together subsequently three diffeological spaces; in Section~\ref{other:glue:config:sect} we consider the alternative possibilities.

\subsubsection{Invariance with respect to gluing maps and iterated gluings}\label{invarnce:sect} 

By Corollary~\ref{omega:of:glued:cor} the spaces $\Omega^m((X_1\cup_fX_2)\cup_{\tilde{g}_1}X_3)$ and $\Omega^m((X_1\cup_fX_2)\cup_{\tilde{g}_2}X_3)$ are diffeomorphic to $\Omega^m(X_1\cup_fX_2)_{\tilde{g}_1}\oplus_{comp}\Omega^m(X_3)$ and $\Omega^m(X_1\cup_fX_2)_{\tilde{g}_2}\oplus_{comp}\Omega^m(X_3)$ respectively, with the compatibilities being with respect to $\tilde{g}_1$ and $\tilde{g}_2$. In this section we consider the spaces $\Omega^m(X_1\cup_fX_2)_{\tilde{g}_1}$ and $\Omega^m(X_1\cup_fX_2)_{\tilde{g}_2}$; for this, we need some further consideration and notions.

\paragraph{$\tilde{g}_1$-equivalent plots on $X_1\cup_fX_2$} We first introduce the following notion.

\begin{defn}
Two plots $p_1,p_1':U\to X_1$ are said to be \textbf{$f\cup g_1$-equivalent} if for all $u\in U$ the inequality $p_1(u)\neq p_1'(u)$ implies that $p_1(u),p_1'(u)\in Y\cup Z_1$ and the following are true:
\begin{itemize}
\item if $p_1(u),p_1'(u)\in Z_1$ then $g_1(p_1(u))=g_1(p_1'(u))$;
\item if $p_1(u)\in Y\setminus Z_1$ and $p_1'(u)\in Z_1$ then there exists $z\in Y\cap Z_1$ such that $f(z)=f(p_1(u))$ and $g_1(z)=g_1(p_1'(u))$;
\item symmetrically, if $p_1(u)\in Z_1$ and $p_1'(u)\in Y\setminus Z_1$ then there exists $z'\in Y\cap Z_1$ such that $f(z')=f(p_1'(u))$ and $g_1(z')=g_1(p_1(u))$;
\item if $p_1(u),p_1'(u)\in Y\setminus Z_1$ then either $f(p_1(u))=f(p_1'(u))$ or there exist $z,z'\in Y\cap Z_1$ such that $f(z)=f(p_1(u))$, $f(z')=f(p_1'(u))$, and $g_1(z)=g_1(z')$.
\end{itemize}
A form $\omega_1\in\Omega^m(X_1)$ is said to be \textbf{$f\cup g_1$-invariant} if for any two $f\cup g_1$-equivalent plots $p_1,p_1'$ we have $\omega_1(p_1)=\omega_1(p_1')$.
\end{defn}

\begin{example}
In Example 4, observe first that, as required, $f$ is a subduction and $g_1$ is such that $f(x_0,x_1)=f(x_0',x_1')$ $\Rightarrow$ $g_1(x_0,x_1)=g_1(x_0',x_1')$. We have that two plots $F_1,F_2$ of $X_1$ with the same domain of definition $U$ are $f\cup g_1$-equivalent iff $F_1(U_{\neq})=F_2(U_{\neq})\subseteq\{x_0x_1=0\}$ and for all $u\in U_{\neq}$ one of the following holds:
\begin{itemize}
\item $F_1^1(u)=F_2^1(u)=0$, $F_1^0(u),F_2^0(u)\in(0,\pi)$, and $F_2^0(u)=\pi-F_1^0(u)$;
\item $F_1^0(u)=F_2^1(u)=0$, $F_1^1(u),F_2^0(u)\in(0,\pi)$, and $F_2^0(u)=\pi-F_1^1(u)$;
\item $F_1^1(u)=F_2^0(u)=0$, $F_1^0(u),F_2^1(u)\in(0,\pi)$, and $F_2^1(u)=\pi-F_1^0(u)$;
\item $F_1^0(u)=F_2^0(u)=0$, $F_1^1(u),F_2^1(u)\in(0,\pi)$, and $F_2^1(u)=\pi-F_1^1(u)$.
\end{itemize}
In particular, there exist $f\cup g_1$-equivalent plots that are neither $f$-equivalent nor $g_1$-equivalent, for instance, this is the case of the pair $(0,\pi)\ni t\mapsto(t,0)$, $(0,\pi)\ni t\mapsto(0,\pi-t)$. It now follows that a $0$-form $h_1$ on $X_1$ is $f\cup g_1$-invariant iff $h_1(t,0)=h_1(0,t)\,\,\forall\,t>0$ and $h_1(t,0)=h_1(\pi-t,0)\,\,\forall t\in(0,\pi)$, and a $1$-form $h_{10}dx_0+h_{11}dx_1$ is $f\cup g_1$-invariant iff $h_{10}(t,0)=h_{11}(0,t)\,\,\forall\,t>0$ and $h_{10}(t,0)=-h_{10}(\pi-t,0)\,\,\forall t\in(0,\pi)$.
\end{example}

The definition of $f\cup g_1$-equivalent plots is in general quite cumbersome, but in certain specific cases it reduces to a much simpler case, as the following example shows.

\begin{example}\label{example5:ex}
Let $X_1,X_3$ be two-dimensional corners, let $X_2$ be a one-dimensional corner, let $Y,Z_1\subset X_1$ be given, respectively, by the equations $x_0x_1=0$ and $x_1=0$, and let $f:Y\to X_2$, $g_1:Z_1\to X_3$ be given by $f(t,0)=f(0,t)=t$, $g_1(t,0)=(0,t)$. In this case, it is easy to see that (due to injectivity of $g_1$) the $f\cup g_1$-equivalent plots are precisely the $f$-equivalent plots (which were already described in Example~\ref{f-equiv:ex}).
\end{example}

The space of all $f\cup g_1$-invariant forms will be denoted by $\Omega_{f\cup g_1}^m(X_1)$. We establish first of all the following property. 

\begin{lemma}
Any two $f$-equivalent plots are also $f\cup g_1$-equivalent.
\end{lemma}

\begin{proof}
Let $p_1,p_1':U\to X_1$ be two $f$-equivalent plots, and let $u$ be such that $p_1(u)\neq p_1'(u)$. Then $p_1(u),p_1'(u)\in Y\subseteq Y\cup Z_1$, and we have:
\begin{itemize}
\item if $p_1(u),p_1'(u)\in Z_1$ then $g_1(p_1(u))=g_1(p_1'(u))$ by $f$-equivalence and the assumption on $g_1$;
\item let $p_1(u)\in Y\setminus Z_1$ and $p_1'(u)\in Z_1$, and let $z=p_1'(u)\in Y\cap Z_1$. Then $f(z)=f(p_1(u))$ by $f$-equivalence, and $g_1(z)=g_1(p_1'(u))$ is trivial;
\item likewise, let $p_1(u)\in Z_1$ and $p_1'(u)\in Y\setminus Z_1$. Then we set $z'=p_1(u)$, obtaining $f(z')=f(p_1'(u))$ by $f$-equivalence and $g_1(z')=g_1(p_1(u))$ trivially;
\item finally, if $p_1(u),p_1'(u)\in Y\setminus Z_1$ then $f(p_1(u))=f(p_1'(u))$.
\end{itemize}
We thus obtain the desired conclusion.
\end{proof}

The following result clarifies now the meaning of the notion of $f\cup g_1$-equivalent plots.

\begin{lemma}\label{lem:2:25}
Two plots $p_1,p_1':U\to X_1$ are $f\cup g_1$-equivalent if and only if $\alpha_1\circ p_1,\alpha_1\circ p_1'$ are $\tilde{g}_1$-equivalent.
\end{lemma}

\begin{proof}
Let $p_1,p_1':U\to X_1$ be two $f\cup g_1$-equivalent plots, and let $u\in U$ be such that $\alpha_1(p_1(u))\neq\alpha_1(p_1'(u))$. Observe that this immediately implies that $p_1(u)\neq p_1'(u)$ and $f(p_1(u))\neq f(p_1'(u))$, therefore it follows from the definition of $f\cup g_1$-equivalent plots that $\alpha_1(p_1(u)),\alpha_1(p_1'(u))\in\alpha_1(Z_1)$. Consider now $\tilde{g}_1(\alpha_1(p_1(u)))$ and $\tilde{g}_1(\alpha_1(p_1'(u)))$. Since $f(p_1(u))\neq f(p_1'(u))$, it follows from the definition of $f\cup g_1$-equivalent plots that there always exist $z,z'\in Y\cap Z_1$, where we might have $z=p_1(u)$ and/or $z'=p_1'(u)$, such that $f(z)=f(p_1(u))$, $f(z')=f(p_1'(u))$, and $g_1(z)=g_1(z')$. Then it follows from the definition of $\tilde{g}_1$ and the assumption on $g_1$ that $\tilde{g}_1(\alpha_1(p_1(u)))=g_1(z)=g_1(z')=\tilde{g}_1(\alpha_1(p_1'(u)))$, hence $\alpha_1\circ p_1,\alpha_1\circ p_1'$ are $\tilde{g}_1$-equivalent.

Suppose now that $p_1,p_1':U\to X_1$ are two plots of $X_1$ such that $\alpha_1\circ p_1,\alpha_1\circ p_1'$ are $\tilde{g}_1$-equivalent, and let $u\in U$ be such that $p_1(u)\neq p_1'(u)$. Suppose first that $\alpha_1(p_1(u))=\alpha_1(p_1'(u))$; then $p_1(u),p_1'(u)\in Y\subseteq Y\cup Z_1$ by definition of map $\alpha_1$. Suppose now that $\alpha_1(p_1(u))\neq\alpha_1(p_1'(u))$; then $\alpha_1(p_1(u)),\alpha_1(p_1'(u))\in\alpha_1(Z_1)$, which immediately implies that $p_1(u),p_1'(u)\in Y\cup Z_1$. 

Now, keeping in mind that $\tilde{g}_1(\alpha_1(p_1(u)))=\tilde{g}_1(\alpha_1(p_1'(u)))$, we consider the following four cases:
\begin{itemize}
\item let $p_1(u),p_1'(u)\in Z_1$. Then the equality $\tilde{g}_1(\alpha_1(p_1(u)))=\tilde{g}_1(\alpha_1(p_1'(u)))$ implies that $g_1(p_1(u))=g_1(p_1'(u))$ as needed by definition of the map $\tilde{g}_1$;
\item let $p_1(u)\in Y\setminus Z_1$ and $p_1'(u)\in Z_1$. Since $\alpha_1(p_1(u))\in\alpha_1(Z_1)$, there exists $z\in Y\cap Z_1$ such that $f(z)=f(p_1(u))$ (that is, $\alpha_1(z)=\alpha_1(p_1(u))$). Thus, as before, $g_1(z)=\tilde{g}_1(\alpha_1(p_1(u)))=\tilde{g}_1(\alpha_1(p_1'(u)))=g_1(p_1'(u))$, as wanted;
\item let $p_1(u)\in Z_1$ and $p_1'(u)\in Y\setminus Z_1$. This case is considered exactly as the previous one;
\item finally, let $p_1(u),p_1'(u)\in Y\setminus Z_1$. If $\alpha_1(p_1(u))=\alpha_1(p_1'(u))$, we immediately obtain that $f(p_1(u))=f(p_1'(u))$; otherwise, since $p_1(u),p_1'(u)\in\alpha_1(Z_1)$, there exist $z,z'\in Y\cap Z_1$ such that $f(z)=f(p_1(u))$ and $f(z')=f(p_1'(u))$, so that as before $g_1(z)=\tilde{g}_1(\alpha_1(p_1(u)))=\tilde{g}_1(\alpha_1(p_1'(u)))=g_1(z')$.
\end{itemize}
We thus conclude that $p_1,p_1'$ are $f\cup g_1$-equivalent.
\end{proof}

Next, we need the following notion.

\begin{defn} 
Two plots $p_2,p_2':U\to X_2$ are said to be \textbf{$(\cup_f,g_1)$-equivalent} if for all $u\in U$ such that $p_2(u)\neq p_2'(u)$ we have that $p_2(u),p_2'(u)\in f(Y\cap Z_1)$, and there exist $z,z'\in Y\cap Z_1$ such that $f(z)=p_2(u)$, $f(z')=p_2'(u)$ and $g_1(z)=g_1(z')$. A form $\omega_2\in\Omega^m(X_2)$ is said to be \textbf{$(\cup_f,g_1)$-invariant} if for any two $(\cup_f,g_1)$-equivalent plots $p_2,p_2'$ we have $\omega_2(p_2)=\omega_2(p_2')$.
\end{defn}

\begin{example}
In Example 4, two plots $F_1,F_2$ of $X_2$ with the same domain of definition $U$ are $(\cup_f,g_1)$-equivalent iff $F_1(U_{\neq})=F_2(U_{\neq})\subset(0,\pi)$ and for all $u\in U_{\neq}$ we have $F_2(u)=\pi-F_1(u)$. It then follows that a $0$-form $h_2$ on $X_2$ is $(\cup_f,g_1)$-invariant iff $h_2(t)=h_2(\pi-t)$ $\forall t\in(0,\pi)$, and a $1$-form $h_2dt$ is $(\cup_f,g_1)$-invariant iff $h_2(t)=-h_2(\pi-t)$ $\forall t\in(0,\pi)$.
\end{example}

Again, the notion of $(\cup_f,g_1)$-equivalent plots becomes quite simpler in certain specific cases.

\begin{example}
In the case of Example~\ref{example5:ex}, any plot of $X_2$ is $(\cup_f,g_1)$-equivalent.
\end{example}

The space of all $(\cup_f,g_1)$-invariant forms in $\Omega^m(X_2)$ will be denoted by $\Omega_{\cup_f,g_1}^m(X_2)$. We can now obtain the following result.

\begin{lemma}\label{lem:2:27}
Two plots $p_2,p_2':U\to X_2$ are $(\cup_f,g_1)$-equivalent if and only if $\alpha_2\circ p_2,\alpha_2\circ p_2'$ are $\tilde{g}_1$-equivalent.
\end{lemma}

\begin{proof}
Let $p_2,p_2'$ be $(\cup_f,g_1)$-equivalent, and let $u$ be such that $\alpha_2(p_2(u))\neq \alpha_2(p_2'(u))$. Then of course $p_2(u)\neq p_2'(u)$, and by definition of $(\cup_f,g_1)$-equivalence there exist $z,z'\in Y\cap Z_1$ such that $f(z)=p_2(u)$, $f(z')=p_2'(u)$ and $g_1(z)=g_1(z')$. Observe first that $\alpha_1(z)=\alpha_2(f(z))$, $\alpha_1(z')=\alpha_2(f(z'))$ by construction; then we immediately obtain that $\alpha_2(p_2(u)),\alpha_2(p_2'(u))\in\alpha_1(Z_1)$. Furthermore, since $\alpha_1(z)=\alpha_2(p_2(u))$, $\alpha_1(z')=\alpha_2(p_2'(u))$, we have $\tilde{g}_1(\alpha_2(p_2(u)))=g_1(z)$, $\tilde{g}_1(\alpha_2(p_2'(u)))=g_1(z')$ by definition of $\tilde{g}_1$. Whence the equality $\tilde{g}_1(\alpha_2(p_2(u)))=\tilde{g}_1(\alpha_2(p_2'(u)))$.

Let now $p_2,p_2'$ be such that $\alpha_2\circ p_2,\alpha_2\circ p_2'$ are $\tilde{g}_1$-equivalent, and let $p_2(u)\neq p_2'(u)$. Since $\alpha_2$ is injective, $\alpha_2(p_2(u))\neq \alpha_2(p_2'(u))$, hence $\alpha_2(p_2(u)),\alpha_2(p_2'(u))\in\alpha_1(Z_1)$, and so $p_2(u),p_2'(u)\in f(Y\cap Z_1)$ by construction. In particular, there exist $z,z'\in Y\cap Z_1$ such that $f(z)=p_2(u)$, $f(z')=p_2'(u)$. Observe again that $\alpha_1(z)=\alpha_2(p_2(u))$ and $\alpha_1(z')=\alpha_2(p_2'(u))$, and so $g_1(z)=\tilde{g}_1(\alpha_2(p_2(u)))=\tilde{g}_1(\alpha_2(p_2'(u)))=g_1(z')$, hence $p_2,p_2'$ are $(\cup_f,g_1)$-equivalent.
\end{proof}

We also have the following observation.

\begin{lemma}\label{lem:2:28}
Let $p_1,p_1'$ are two plots of $X_1$ whose images are contained in $Y$. Then $p_1,p_1'$ are $f\cup g_1$-equivalent if and only if $f\circ p_1,f\circ p_1'$ are $(\cup_f,g_1)$-equivalent.
\end{lemma}

\begin{proof}
Let $p_1,p_1'$ be $f\cup g_1$-equivalent, and suppose that $f(p_1(u))\neq f(p_1'(u))$. Then of course $p_1(u)\neq p_1'(u)$, so that $p_1(u),p_1'(u)\in Y\cup Z_1$. Suppose first that $p_1(u),p_1'(u)\in Y\setminus Z_1$; then by $f\cup g_1$-equivalence there exist $z,z'\in Y\cap Z_1$ such that $f(z)=f(p_1(u))$, $f(z')=f(p_1'(u))$ and $g_1(z)=g_1(z')$, as wanted. Suppose now that $p_1(u)\in Y\setminus Z_1$ and $p_1'(u)\in Z_1$. Then $p_1'(u)\in Y\cap Z_1$, and by $f\cup g_1$-equivalence there exists $z\in Y\cap Z_1$ such that $f(z)=f(p_1(u))$ and $g_1(z)=g_1(p_1'(u))$, and it suffices to set $z'=p_1'(u)$. The case $p_1(u)\in Z_1$, $p_1'(u)\in Y\setminus Z_1$ is of course completely analogous. Finally, suppose $p_1(u),p_1'(u)\in Z_1$; then $p_1(u),p_1'(u)\in Y\cap Z_1$, and we can set $z=p_1(u)$, $z'=p_1'(u)$. Thus, $f\circ p_1$, $f\circ p_1'$ are $(\cup_f,g_1)$-equivalent.

\emph{Vice versa}, suppose $p_1,p_1'$ are such that $f\circ p_1$, $f\circ p_1'$ are $(\cup_f,g_1)$-equivalent, and let $p_1(u)\neq p_1'(u)$. Recall that we have $p_1(u),p_1'(u)\in Y\subseteq Y\cup Z_1$ by choice of $p_1,p_1'$. Suppose first that $p_1(u),p_1'(u)\in Y\setminus Z_1$. If $f(p_1(u))=f(p_1'(u))$, there is nothing to prove; if $f(p_1(u))\neq f(p_1'(u))$ then by $(\cup_f,g_1)$-equivalence  $f(p_1(u)),f(p_1'(u))\in f(Y\cap Z_1)$ and there exist $z,z'\in Y\cap Z_1$ such that $f(z)=f(p_1(u))$, $f(z')=f(p_1'(u))$ and $g_1(z)=g_1(z')$, as needed. Suppose now that $p_1(u)\in Y\setminus Z_1$ and $p_1'(u)\in Z_1$, hence $p_1'(u)\in Y\cap Z_1$. If $f(p_1(u))=f(p_1'(u))$, it suffices to set $z=p_1'(u)$ to obtain the condition needed for $f\cup g_1$-equivalence. If $f(p_1(u))\neq f(p_1'(u))$, by $(\cup_f,g_1)$-equivalence there exist $z,z'\in Y\cap Z_1$ such that $f(z)=f(p_1(u))$, $f(z')=f(p_1'(u))$ and $g_1(z)=g_1(z')$. Since by the assumption on $g_1$ $f(z')=f(p_1'(u))$ implies $g_1(z')=g_1(p_1'(u))$, the $z$ is the one required by the definition of $f\cup g_1$-equivalence. The case $p_1(u)\in Z_1$, $p_1'(u)\in Y\setminus Z_1$ is of course completely analogous. Finally, let $p_1(u),p_1'(u)\in Z_1$, then by assumption $p_1(u),p_1'(u)\in Y\cap Z_1$. If $f(p_1(u))=f(p_1'(u))$, we also have $g_1(p_1(u))=g_1(p_1'(u))$. If $f(p_1(u))\neq f(p_1'(u))$ then by $(\cup_f,g_1)$-equivalence there exist $z,z'\in Y\cap Z_1$ such that $f(z)=f(p_1(u))$ (so that $g_1(z)=g_1(p_1(u))$), $f(z')=f(p_1'(u))$ (so that $g_1(z')=g_1(p_1'(u))$) and $g_1(z)=g_1(z')$, and we immediately conclude that $g_1(p_1(u))=g_1(p_1'(u))$. Thus, $p_1,p_1'$ are $f\cup g_1$-equivalent.
\end{proof}

\paragraph{$\tilde{g}_2$-equivalent plots on $X_1\cup_fX_2$} Next, we consider the following notion.

\begin{defn}
Two plots $p_1,p_1':U\to X_1$ are said to be \textbf{$(\cup_f,g_2)$-equivalent} if for all $u\in U$ such that $p_1(u)\neq p_1'(u)$ either $p_1(u),p_1'(u)\in Y$ and  $f(p_1(u))=f(p_1'(u))$ or $p_1(u),p_1'(u)\in f^{-1}(Z_2)$ and $g_2(f(p_1(u)))=g_2(f(p_1'(u)))$. A form $\omega_1\in\Omega^m(X_1)$ is said to be \textbf{$(\cup_f,g_2)$-invariant} if for any two $(\cup_f,g_2)$-equivalent plots $p_1,p_1'$ we have $\omega_1(p_1)=\omega_1(p_1')$.
\end{defn}

\begin{example}
In Example 4, two plots $F_1,F_2$ of $X_1$ with the same domain of definition $U$ are $(\cup_f,g_2)$-equivalent if $F_1(U_{\neq})=F_2(U_{\neq})\subseteq\{x_0x_1=0\}$ and for all $u\in U_{\neq}$ one of the following occurs:
\begin{itemize}
\item $F_1^0(u)F_1^1(u)=F_2^0(u)F_2^1(u)=0$ and $F_2^0(u)=F_1^1(u)$, $F_2^1(u)=F_1^0(u)$;
\item $F_1^0(u)=0$, $F_1^1(u)\in(0,2)$, $F_2^0(u)=0$, and $F_2^1(u)=2-F_1^1(u)$;
\item $F_1^0(u)=0$, $F_1^1(u)\in(0,2)$, $F_2^1(u)=0$, and $F_2^0(u)=2-F_1^1(u)$;
\item $F_1^0(u)\in(0,2)$, $F_1^1(u)=0$, $F_2^0(u)=0$, and $F_2^1(u)=2-F_1^0(u)$;
\item $F_1^0(u)\in(0,2)$, $F_1^1(u)=0$, $F_2^1(u)=0$, and $F_2^0(u)=2-F_1^0(u)$.
\end{itemize}
It then follows that a $0$-form $h_1$ on $X_1$ is $(\cup_f,g_2)$-invariant iff $h_1(t,0)=h_1(0,t)\,\,\forall t>0$ and $h_1(t,0)=h_1(2-t,0)\,\,\forall t\in(0,2)$, and a $1$-form $h_{10}dx_0+h_{11}dx_1$ is $(\cup_f,g_2)$-invariant iff $h_{10}(t,0)=h_{11}(0,t)\,\,\forall t>0$ and $h_{10}(t,0)=-h_{10}(2-t,0)\,\,\forall t\in(0,2)$.
\end{example}

The space of all $(\cup_f,g_2)$-invariant forms in $\Omega^m(X_1)$ will be denoted by $\Omega_{\cup_f,g_2}^m(X_1)$. Notice that, since any two $f$-equivalent plots are automatically $(\cup_f,g_2)$-equivalent, we have an \emph{a priori} proper inclusion $\Omega_{\cup_f,g_2}^m(X_1)\subseteq\Omega_f^m(X_1)$. In the case when $f(Y)\cap Z_2=\emptyset$, this inclusion is of course an equality. 

\begin{lemma}\label{lem:2:30}
Two plots $p_1,p_1':U\to X_1$ are \textbf{$(\cup_f,g_2)$-equivalent} if and only if $\alpha_1\circ p_1,\alpha_1\circ p_1'$ are $\tilde{g}_2$-equivalent.
\end{lemma}

\begin{proof}
Let $p_1,p_1'$ be $(\cup_f,g_2)$-equivalent, and let $\alpha_1(p_1(u))\neq\alpha_1(p_1'(u))$. Then in particular $f(p_1(u))\neq f(p_1'(u))$, and hence $p_1(u),p_1'(u)\in f^{-1}(Z_2)$, and so $\alpha_1(p_1(u)),\alpha_1(p_1'(u))\in \alpha_2(Z_2)$. Let us now consider $\tilde{g}_2(\alpha_1(p_1(u)))$ and $\tilde{g}_2(\alpha_1(p_1'(u)))$. Since $p_1(u)\in f^{-1}(Z_2)$, we have $\alpha_1(p_1(u))=\alpha_2(f(p_1(u)))$, and so $\tilde{g}_2(\alpha_1(p_1(u)))=g_2(f(p_1(u)))$ by definition of $\tilde{g}_2$; likewise, $\tilde{g}_2(\alpha_1(p_1'(u)))=g_2(f(p_1'(u)))$. Since $g_2(f(p_1(u)))=g_2(f(p_1'(u)))$ by $(\cup_f,g_2)$-equivalence, we obtain that $\tilde{g}_2(\alpha_1(p_1(u)))=\tilde{g}_2(\alpha_1(p_1'(u)))$, so $\alpha_1\circ p_1,\alpha_1\circ p_1'$ are $\tilde{g}_2$-equivalent.

\emph{Vice versa}, let $p_1,p_1':U\to X_1$ be such that $\alpha_1\circ p_1,\alpha_1\circ p_1'$ are $\tilde{g}_2$-equivalent, and let $p_1(u)\neq p_1'(u)$. Suppose first that $\alpha_1\circ p_1=\alpha_1\circ p_1'$. Then by definition of the map $\alpha_1$ we have $p_1(u),p_1'(u)\in Y$ and $f(p_1(u))=f(p_1'(u))$. Suppose now that $\alpha_1\circ p_1\neq\alpha_1\circ p_1'$. Then $\alpha_1\circ p_1,\alpha_1\circ p_1'\in \alpha_2(Z_2)$, and so, since $\alpha_2$ is injective and $\alpha_1|_Y=\alpha_2\circ f$, we have $p_1(u),p_1'(u)\in f^{-1}(Z_2)$ and $\alpha_1(p_1(u))=\alpha_2(f(p_1(u)))$, $\alpha_1(p_1'(u))=\alpha_2(f(p_1'(u)))$. Thus the equality $\tilde{g}_2(\alpha_1(p_1(u)))=\tilde{g}_2(\alpha_1(p_1'(u)))$ implies $g_2(f(p_1(u)))=g_2(f(p_1'(u)))$ by definition of $\tilde{g}_2$. Hence $p_1,p_1'$ are $(\cup_f,g_2)$-equivalent.
\end{proof}

Note that we also have the following observation.

\begin{lemma}\label{lem:2:31}
Let $p_1,p_1'$ be two plots of $X_1$ whose images are contained in $Y$. Then $p_1,p_1'$ are $(\cup_f,g_2)$-equivalent if and only if $f\circ p_1,f\circ p_1'$ are $g_2$-equivalent.
\end{lemma}

\begin{proof}
Suppose that $p_1,p_1'$ are $(\cup_f,g_2)$-equivalent, and let $f(p_1(u))\neq f(p_1'(u))$. Then $f(p_1(u)),f(p_1'(u))\in Z_2$ and $g_2(f(p_1(u)))=g_2(f(p_1'(u)))$ by definition of $(\cup_f,g_2)$-equivalence.

\emph{Vice versa}, let $p_1,p_1'$ be such that $f\circ p_1,f\circ p_1'$ are $g_2$-equivalent, and let $p_1(u)\neq p_1'(u)$. If $f(p_1(u))=f(p_1'(u))$, there is nothing to prove. Suppose $f(p_1(u))\neq f(p_1'(u))$; then $f(p_1(u)),f(p_1'(u))\in Z_2$ and $g_2(f(p_1(u)))=g_2(f(p_1'(u)))$, hence $p_1,p_1'$ are $(\cup_f,g_2)$-equivalent.
\end{proof}

Finally, we will need the following statement.

\begin{lemma}\label{lem:2:32}
Two plots $p_2,p_2':U\to X_2$ are $g_2$-equivalent if and only if $\alpha_2\circ p_2,\alpha_2\circ p_2'$ are $\tilde{g}_2$-equivalent.
\end{lemma}

\begin{proof}
Let $p_2,p_2'$ be $g_2$-equivalent, and suppose $\alpha_2(p_2(u))\neq \alpha_2(p_2'(u))$. Then $p_2(u)\neq p_2'(u)$, hence $p_2(u),p_2'(u)\in Z_2$ and $g_2(p_2(u))=g_2(p_2'(u))$, therefore $\alpha_2(p_2(u)),\alpha_2(p_2'(u))\in \alpha_2(Z_2)$ and $\tilde{g}_2(\alpha_2(p_2(u)))=\tilde{g}_2(\alpha_2(p_2'(u)))$, so $\alpha_2\circ p_2,\alpha_2\circ p_2'$ are $\tilde{g}_2$-equivalent. Since $\alpha_2$ is invertible, the \emph{vice versa} can be obtained by reversing the order of the argument.
\end{proof}

\paragraph{Invariant forms on $X_1\cup_fX_2$} We now establish the main result of this section.

\begin{thm}\label{f,g-inv:iterat:thm}
There are the following diffeomorphisms: $$\Omega_{\tilde{g}_1}^m(X_1\cup_fX_2)\cong\Omega_{f\cup g_1}^m(X_1)\oplus_{comp}\Omega_{\cup_f,g_1}^m(X_2)$$ and $$\Omega_{\tilde{g}_2}^m(X_1\cup_fX_2)\cong\Omega_{\cup_f,g_2}^m(X_1)\oplus_{comp}\Omega_{g_2}^m(X_2),$$ where in both cases the compatibility is with respect to the map $f$. The two diffeomorphisms are given by the appropriate restrictions of the map $\pi^*$.
\end{thm}

\begin{proof}
Let us prove the existence of the first diffeomorphism. Let $\omega=\omega_1\cup_f\omega_2\in\Omega_{\tilde{g}_1}^m(X_1\cup_fX_2)$, let $p_1,p_1'$ be two $f\cup g_1$-equivalent plots of $X_1$, and let $p_2,p_2'$ be $(\cup_f,g_1)$-equivalent plots of $X_2$. Then by Lemmas~\ref{lem:2:25} and \ref{lem:2:27} $\alpha_1\circ p_1,\alpha_1\circ p_1'$ are $\tilde{g}_1$-equivalent, and so are $\alpha_2\circ p_2,\alpha_2\circ p_2'$. Thus, $\omega_1(p_1)=\omega(\alpha_1\circ p_1)=\omega(\alpha_1\circ p_1')=\omega_1(p_1')$, and $\omega_2(p_2)=\omega(\alpha_2\circ p_2)=\omega(\alpha_2\circ p_2')=\omega_2(p_2')$. Hence $\omega_1\in\Omega_{f\cup g_1}^m(X_1)$ and $\omega_2\in\Omega_{\cup_f,g_1}^m(X_2)$.

\emph{Vice versa}, let $\omega_1\in\Omega_{f\cup g_1}^m(X_1)$, $\omega_2\in\Omega_{\cup_f,g_1}^m(X_2)$ be $f$-compatible, and let $\omega=\omega_1\cup_f\omega_2$. Let $p,p'$ be $\tilde{g}_1$-equivalent plots of $X_1\cup_fX_2$; assume, as we can, that their domain of definition is small enough so that each of them lifts either to a plot of $X_1$ or one of $X_2$. Suppose first that both $p$ and $p'$ lift to plots $p_1,p_1'$ of $X_1$, $p=\alpha_1\circ p_1$, $p'=\alpha_1\circ p_1'$. Then by Lemma~\ref{lem:2:25} $p_1,p_1'$ are $f\cup g_1$-equivalent, and so $\omega(p)=\omega_1(p_1)=\omega_1(p_1')=\omega(p')$. Likewise, suppose that both $p,p'$ lift to plots $p_2,p_2'$ of $X_2$, $p=\alpha_2\circ p_2$, $p'=\alpha_2\circ p_2'$; then by Lemma~\ref{lem:2:27} $p_2,p_2'$ are $(\cup_f,g_1)$-equivalent, and we have $\omega(p)=\omega_2(p_2)=\omega_2(p_2')=\omega(p')$. Finally, suppose that one of the plots, say $p$, lifts to a plot $p_1$ of $X_1$, $p=\alpha_1\circ p_1$, while the other lifts to a plot $p_2'$ of $X_2$, $p'=\alpha_2\circ p_2'$. Then it follows from $\tilde{g}_1$-equivalence that the image of $p_2'$ is contained in $f(Y)$. Since $f$ is a subduction, there then locally exists a plot $p_1'$ of $Y$ such that $p_2'=f\circ p_1'$. Since $p'=\alpha_2\circ p_2'$ and $\alpha_1|_Y=\alpha_2\circ f$, we have $p'=\alpha_1\circ p_1'$. Thus, by Lemma~\ref{lem:2:25} $p_1,p_1'$ are $f\cup g_1$-equivalent, and again $\omega(p)=\omega_1(p_1)=\omega_1(p_1')=\omega(p')$. Thus, the existence of the first diffeomorphism is proven.

Let us now prove the existence of the second diffeomorphism. Let $\omega=\omega_1\cup_f\omega_2\in\Omega_{\tilde{g}_2}$, let $p_1,p_1'$ be two $(\cup_f,g_2)$-equivalent plots, and let $p_2,p_2'$ be two $g_2$-equivalent plots. Then $\omega_1(p_1)=\omega(\alpha_1\circ p_1)=\omega(\alpha_1\circ p_1')=\omega_1(p_1')$ and $\omega_2(p_2)=\omega(\alpha_2\circ p_2)=\omega(\alpha_2\circ p_2')=\omega_2(p_2')$, so $\omega_1\in\Omega_{\cup_f,g_2}^m(X_1)$ and $\omega_2\in\Omega_{g_2}^m(X_2)$.

\emph{Vice versa}, let $\omega_1\in\Omega_{\cup_f,g_2}^m(X_1)$, $\omega_2\in\Omega_{g_2}^m(X_2)$ be $f$-compatible, and let $\omega=\omega_1\cup_f\omega_2$. Let $p,p'$ be two $\tilde{g}_2$-equivalent plots of $X_1\cup_fX_2$; assume again that their common domain of definition is such that each of them lifts either to a plot of $X_1$ or one of $X_2$. Suppose that both $p,p'$ lift to plots $p_1,p_1'$ of $X_1$, $p=\alpha_1\circ p_1$, $p'=\alpha_1\circ p_1'$. Then by Lemma~\ref{lem:2:30} $p_1,p_1'$ are $(\cup_f,g_2)$-equivalent, and we have $\omega(p)=\omega_1(p_1)=\omega_1(p_1')=\omega(p')$. Likewise, suppose that both $p,p'$ lift to plots $p_2,p_2'$ of $X_2$, $p=\alpha_2\circ p_2$, $p'=\alpha_2\circ p_2'$; then by Lemma~\ref{lem:2:32} $p_2,p_2'$ are $g_2$-equivalent, and we have $\omega(p)=\omega_2(p_2)=\omega_2(p_2')=\omega(p')$. Finally, suppose that one of the plots, say $p$, lifts to a plot $p_1$ of $X_1$, $p=\alpha_1\circ p_1$, while the other lifts to a plot $p_2'$ of $X_2$, $p'=\alpha_2\circ p_2'$. Then it follows from $\tilde{g}_2$-equivalence of $p,p'$ that the image of $p_1$ is contained in $Y$, and since $f$ is smooth, $p_2=f\circ p_1$ is a plot of $X_2$. Since $\alpha_1|_Y=\alpha_2\circ f$, we have $p=\alpha_2\circ p_2$, hence by Lemma~\ref{lem:2:32} $p_2,p_2'$ are $g_2$-equivalent, and so we have again $\omega(p)=\omega_2(p_2)=\omega_2(p_2')=\omega(p')$. All cases having been exhausted, the theorem is proven.
\end{proof}

The theorem just proven indicates that the behavior of the invariance under gluing maps is in general rather cumbersome under iterated gluings.

\subsubsection{Compatibility and iterated gluings} 

Let now consider the behavior of compatibility under iterated gluings. Let $\omega_i\in\Omega^m(X_i)$ for $i=1,2,3$, where $\omega_1$ and $\omega_2$ are $f$-compatible (\emph{i.e.}, such that $i^*\omega_1=f^*j^*\omega_2$); recall the natural inclusions $\bar{k}_1$, $\bar{k}_2$, and the equalities $\alpha_1\circ k_1=\bar{k}_1\circ\alpha_1|_{Z_1}$, $\bar{k_2}=\alpha_2\circ k_2\circ (\alpha_2|_{Z_2})^{-1}$. We first establish the following auxiliary result.

\begin{lemma}
Let $p'$ be a plot of either $\alpha_1(Z_1)$ or of $\alpha_2(Z_2)$. Then:
\begin{enumerate}
\item if $p'$ is a plot of $\alpha_1(Z_1)$ and $p_1$ is its lift to a plot of $Z_1$ then $\bar{k}_1^*(\omega_1\cup_f\omega_2)(p')=k_1^*\omega_1(p_1)$;
\item if $p'$ is a plot of $\alpha_2(Z_2)$ and $p_2$ is its lift to a plot of $Z_2$ then $\bar{k}_2^*(\omega_1\cup_f\omega_2)(p')=k_2^*\omega_2(p_2)$.
\end{enumerate}
\end{lemma}

\begin{proof}
The lemma follows from the following two commutative diagrams:
$$
\begin{CD}
U @>p_1>> Z_1 @>k_1>> X_1\\
@\vert @V\alpha_1|_{Z_1}VV @V\alpha_1VV \\
U @>p'>> \alpha_1(Z_1) @>\bar{k}_1>> X_1\cup_fX_2
\end{CD},\,\,\,\,\,\,\,
\begin{CD}
U @>p_2>> Z_2 @>k_2>> X_2\\
@\vert @V\alpha_2|_{Z_2}VV @V\alpha_2VV \\
U @>p'>> \alpha_2(Z_2) @>\bar{k}_2>> X_1\cup_fX_2
\end{CD}.
$$
\end{proof}

We can now obtain the following.

\begin{prop}\label{comp:iter:prop}
We have:
\begin{enumerate}
\item The form $\omega_1\cup_f\omega_2$ is $\tilde{g}_1$-compatible with $\omega_3$ if and only if $\omega_3$ is $g_1$-compatible with $\omega_1$;
\item The form $\omega_1\cup_f\omega_2$ is $\tilde{g}_2$-compatible with $\omega_3$ if and only if $\omega_3$ is $g_2$-compatible with $\omega_2$.
\end{enumerate}
\end{prop}

\begin{proof}
Let us prove 1. Suppose first that $\omega_1$ and $\omega_3$ are $g_1$-compatible, $k_1^*\omega_1=g_1^*l_1^*\omega_3$; let us show that $\bar{k}_1^*(\omega_1\cup_f\omega_2)=\tilde{g}_1^*l_1^*\omega_3$. Let $p':U\to\alpha_1(Z_1)$ be a plot; we can again assume that $U$ is small enough so that $p'$ lifts either to a plot $p_1'$ of $Z_1$ or to a plot $p_2'$ of $X_2$. In this latter case, furthermore, $p_2'$ takes values in $f(Y)$, and since $f$ is assumed to be a subduction, it locally lifts to a plot of $X_1$ with values in $Y\cap Z_1$; thus, $p'$ always locally lifts to a plot $p_1'$ of $Z_1$, and we have $\bar{k}_1^*(\omega_1\cup_f\omega_2)(p')=k_1^*\omega_1(p_1')=g_1^*l_1^*\omega_3(p_1')$.

Consider now $\tilde{g}_1^*l_1^*\omega_3(p')$. Since $g_1^*=\alpha_1^*\tilde{g}_1^*$, we have $g_1^*l_1^*\omega_3(p_1')=\alpha_1^*\tilde{g}_1^*l_1^*\omega_3(p_1')=\tilde{g}_1^*l_1^*\omega_3(p')$, hence the desired equality $\bar{k}_1^*(\omega_1\cup_f\omega_2)=\tilde{g}_1^*l_1^*\omega_3$.

Suppose now that $\omega_1\cup_f\omega_2$, $\omega_3$ are $\tilde{g}_1$-compatible; let us show that $k_1^*\omega_1=g_1^*l_1^*\omega_3$. Let $p:U\to Z_1$ be a plot of $Z_1$; then of course $p'=\alpha_1\circ p$ is a plot of $\alpha_1(Z_1)$, and we have $k_1^*\omega_1(p)=\bar{k}_1^*(\omega_1\cup_f\omega_2)(p')=\tilde{g}_1^*l_1^*\omega_3(p')$. Finally, we have $g_1^*l_1^*\omega_3(p)=\alpha_1^*\tilde{g}_1^*l_1^*\omega_3(p)=\tilde{g}_1^*l_1^*\omega_3(p')$, which yields the desired equality $k_1^*\omega_1=g_1^*l_1^*\omega_3$.

Let us now prove 2. Suppose first that that $\omega_2$ and $\omega_3$ are $g_2$-compatible, $k_2^*\omega_2=g_2^*l_2^*\omega_3$, and let us show that $\bar{k}_2^*(\omega_1\cup_f\omega_2)=\tilde{g}_2^*l_2\omega_3$. Let $p':U\to \alpha_2(Z_2)$ be a plot with $U$ such that $p'$ lifts either to a plot $p_2'$ of $Z_2$ or a plot $p_1'$ of $X_1$; in this latter case $p_1'$ is a plot of $Y\cap f^{-1}(Z_2)$, and so $p_2'=f\circ p_1'$ is a plot of $Z_2$ that is a lift of $p'$. We thus have in both cases $\bar{k}_2^*(\omega_1\cup_f\omega_2)(p')=k_2^*\omega_2(p_2')=\alpha_2^*g_2^*l_2^*\omega_2(p_2')$ by assumption. On the other hand, $\tilde{g}_2^*l_2^*\omega_3(p')=(\alpha_2^{-1})^*g_2^*l_2^*\omega_3(p')=g_2^*l_2^*\omega_3(p_2')$, hence the desired equality $\bar{k}_2^*(\omega_1\cup_f\omega_2)=\tilde{g}_2^*l_2\omega_3$.

Assume now this latter equality, and let us show that $k_2^*\omega_2=g_2^*l_2^*\omega_3$. Let $p_2:U\to Z_2$ be a plot; denote $p=\alpha_2\circ p_2$. Then $k_2^*\omega_2(p_2)=\bar{k}_2^*(\omega_1\cup_f\omega_2)(p)=\tilde{g}_2^*l_2\omega_3(p)$ by assumption. On the other hand, since $g_2=\tilde{g}_2\circ \alpha_2$, $g_2^*l_2^*\omega_3(p_2)=\alpha_2^*\tilde{g}_2^*l_2^*\omega_3(p_2)=\tilde{g}_2^*l_2\omega_3(p)$, hence the desired equality. 

All cases having been considered, the entire statement is proven.
\end{proof}

We thus obtain the following statement.

\begin{cor}\label{omega:iterat:cor}
The space $\Omega^m((X_1\cup_fX_2)\cup_{\tilde{g}_1}X_3)$ is diffeomorphic to 
\begin{flushleft}
$\{(\omega_1,\omega_2,\omega_3)\,|\,\omega_1\in\Omega_{f\cup g_1}^m(X_1),\omega_2\in\Omega_{\cup_f,g_1}^m(X_2),\omega_3\in\Omega^m(X_3),i^*\omega_1=f^*j^*\omega_2,\,k_1^*\omega_1=g_1^*l_1^*\omega_3\}\leqslant$
\end{flushleft}
\begin{flushright}
$\leqslant\Omega^m(X_1)\oplus\Omega^m(X_2)\oplus\Omega^m(X_3),$
\end{flushright}
and the space $\Omega^m((X_1\cup_fX_2)\cup_{\tilde{g}_2}X_3)$ is diffeomorphic to 
\begin{flushleft}
$\{(\omega_1,\omega_2,\omega_3)\,|\,\omega_1\in\Omega_{\cup_f,g_2}^m(X_1),\omega_2\in\Omega_{g_2}^m(X_2),\omega_3\in\Omega^m(X_3),i^*\omega_1=f^*j^*\omega_2,\,k_2^*\omega_2=g_2^*l_2^*\omega_3\}\leqslant$
\end{flushleft}
\begin{flushright}
$\leqslant\Omega^m(X_1)\oplus\Omega^m(X_2)\oplus\Omega^m(X_3)$.
\end{flushright}
\end{cor}

\begin{proof}
Let us consider $\Omega^m((X_1\cup_fX_2)\cup_{\tilde{g}_1}X_3)$. By Corollary~\ref{omega:of:glued:cor} this space is diffeomorphic to $$\Omega_{\tilde{g}_1}^m(X_1\cup_fX_2)\oplus_{comp}\Omega^m(X_3);$$ by Theorem~\ref{f,g-inv:iterat:thm} the space $\Omega_{\tilde{g}_1}^m(X_1\cup_fX_2)$ is diffeomorphic to $\Omega_{f\cup g_1}^m(X_1)\oplus_{comp}\Omega_{\cup_f,g_1}^m(X_2)$, and the first claim follows from Proposition~\ref{comp:iter:prop}.

Consider now $\Omega^m((X_1\cup_fX_2)\cup_{\tilde{g}_2}X_3)$. By Corollary~\ref{omega:of:glued:cor} it is diffeomorphic to $$\Omega_{\tilde{g}_2}^m((X_1\cup_fX_2)\oplus_{comp}\Omega^m(X_3),$$ and by Theorem~\ref{f,g-inv:iterat:thm} $\Omega_{\tilde{g}_2}^m(X_1\cup_fX_2)$ is diffeomorphic to $\Omega_{\cup_f,g_2}^m(X_1)\oplus_{comp}\Omega_{g_2}^m(X_2)$, so the second claim follows again from Proposition~\ref{comp:iter:prop}.
\end{proof}

\begin{example}
In Example 4, we have that the space $\Omega^0((X_1\cup_fX_2)\cup_{\tilde{g}_1}X_3)$ consists of triples $(h_1,h_2,h_3)$, where $h_1:(0,+\infty)\times(0,+\infty)\to\mathbb{R}$ is a smooth function that can be extended to a smooth $\mathbb{R}^2\to\mathbb{R}$ and $h_2,h_3:(0,+\infty)\to\mathbb{R}$ are smooth functions that can be extended to smooth $\mathbb{R}\to\mathbb{R}$, such that $$h_1(t,0)=h_1(0,t)=h_2(t)\,\,\forall t>0,\,\,\,\,h_1(t,0)=h_3(\sin t)\,\,\forall t\in(0,\pi),$$ the space $\Omega^1((X_1\cup_fX_2)\cup_{\tilde{g}_1}X_3)$ consists of triples $(h_{10}dx_0+h_{11}dx_1,h_2dt,h_3ds)$, where $h_{10},h_{11}:(0,+\infty)\times(0,+\infty)\to\mathbb{R}$ are smooth functions that can be extended to smooth maps $\mathbb{R}^2\to\mathbb{R}$ and $h_2,h_3$ are as above, such that $$h_{10}(t,0)=h_{11}(0,t)=h_2(t)\,\,\forall t>0,\,\,\,\,h_{10}(t,0)=-h_{10}(\pi-t,0)=\cos t\cdot h_3(\sin t)\,\,\forall t\in(0,\pi),$$ and the space $\Omega^2((X_1\cup_fX_2)\cup_{\tilde{g}_1}X_3)$ coincides with $\Omega^2(X_1)$, which in turn coincides with the space of all smooth functions $(0,+\infty)\times(0,+\infty)\to\mathbb{R}$ that can be extended to a smooth map $\mathbb{R}^2\to\mathbb{R}$.

For the space $\Omega^m((X_1\cup_fX_2)\cup_{\tilde{g}_2}X_3)$ we have that $\Omega^0((X_1\cup_fX_2)\cup_{\tilde{g}_2}X_3)$ consists of triples $(h_1,h_2,h_3)$, where $h_1,h_2,h_3$ are as above, such that $$h_1(t,0)=h_1(0,t)=h_2(t)\,\,\forall t>0,\,\,\,\,h_2(t)=h_3((t-1)^2)\,\,\forall t\in(0,2),$$ the space $\Omega^1((X_1\cup_fX_2)\cup_{\tilde{g}_2}X_3)$ consists of triples $(h_{10}dx_0+h_{11}dx_1,h_2dt,h_3ds)$, where $h_{10},h_{11},h_2,h_3$ are as above, such that $$h_{10}(t,0)=h_{11}(0,t)=h_2(t)\,\,\forall t>0,\,\,\,\,h_2(t)=-h_2(2-t)=2(t-1)h_3((t-1)^2)\,\,\forall t\in(0,2),$$ and the space $\Omega^2((X_1\cup_fX_2)\cup_{\tilde{g}_2}X_3)$ coincides with the space of $2$-forms on $X_1$, which in turn coincides with the space of all smooth functions $(0,+\infty)\times(0,+\infty)\to\mathbb{R}$ that can be extended to a smooth map $\mathbb{R}^2\to\mathbb{R}$. Furthermore, using constructions similar to those of Example~\ref{spaces:one:glue:ex}(4) one can see that $\Omega^0((X_1\cup_fX_2)\cup_{\tilde{g}_i}X_3)$, $i=1,2$, is diffeomorphic to a certain vector subspace of $C^\infty([0,+\infty)^2)\oplus C^\infty([0,+\infty))$ (the specific subspaces are given by the above conditions on $h_1,h_2,h_3$) and $\Omega^1((X_1\cup_fX_2)\cup_{\tilde{g}_i}X_3)$, $i=1,2$, is diffeomorphic to a certain vector subspace of $C^\infty([0,+\infty)^2)\oplus C^\infty([0,+\infty)^2)\oplus C^\infty([0,+\infty))$.

In the case of Example~\ref{example5:ex}, recall (Example~\ref{compty:ex}) that two $0$-forms $h_1,h_2$ are $f$-compatible iff $h_1(t,0)=h_1(0,t)=h_2(t)\,\forall t>0$ (so in particular $h_1$ is $f$-invariant), and two $1$-forms $h_{10}dx_0+h_{11}dx_1$, $h_2dt$ are $f$-compatible iff $h_{10}(t,0)=h_{11}(0,t)=h_2(t)$, $h_{10}(0,t)=h_{11}(t,0)$ $\forall t>0$ (so in particular, $h_{10}dx_0+h_{11}dx_1$ is $f$-invariant). It is also easy to see that two $0$-forms $h_1,h_3$ are $g_1$-compatible iff $h_1(t,0)=h_3(0,t)\,\forall t>0$, and two $1$-forms $h_{10}dx_0+h_{11}dx_1$, $h_{30}dy_0+h_{31}dy_1$ are $g_1$-compatible iff $h_{10}(t,0)=h_{31}(0,t)\,\forall t>0$. Thus, the space $$\Omega^0((X_1\cup_fX_2)\cup_{\tilde{g}_1}X_3)$$ consists of triples $(h_1,h_2,h_3)$ such that $$h_1(t,0)=h_1(0,t)=h_2(t)=h_3(0,t)\,\forall t>0,$$ whereas the space $$\Omega^1((X_1\cup_fX_2)\cup_{\tilde{g}_1}X_3)$$ consists of triples $(h_{10}dx_0+h_{11}dx_1,h_2dt,h_{30}dy_0+h_{31}dy_1)$ such that $$h_{10}(t,0)=h_{11}(0,t)=h_2(t)=h_3(0,t),\,\,h_{10}(0,t)=h_{11}(t,0)\,\,\forall t>0.$$
\end{example}

\subsubsection{Condition $\mathcal{D}_f^{i^*}(Y)=\mathcal{D}^{f^*j^*}(Y)$ under iterated gluings} 

In the case of iterated gluings along non-injective maps, we shall need more sophisticated versions of this kind of condition. We introduce a general type of notation $\mathcal{D}_{inv}^{r^*}(T)$ to denote the pushforward by the map $r^*$ of the usual (subset) diffeology on either $\Omega_{inv}^m(S)$ or $\Omega_{inv}^m(S')$, where $r$ is either the inclusion of $T$ into its ambient space $S$ or else it is the composition of a gluing map $h$ defined on $T$ with the inclusion of $r(T)$ into its ambient space $S'$; finally, the lower index $inv$ describes the appropriate invariance condition. Spelling this out, we have:
\begin{center}
\bgroup
\def\arraystretch{1.8}
\begin{tabular}
{|l|c|c|c|}
\hline
Notation & Diffeology on & Map $r^*$ & Pushforward of diffeology on \\ \hline
$\mathcal{D}_{f\cup g_1}^{i^*}(Y)$ & $\Omega^m(Y)$ & $i^*$ & $\Omega_{f\cup g_1}^m(X_1)$\\ 
\hline
$\mathcal{D}_{\cup_f,g_1}^{f^*j^*}(Y)$ & $\Omega^m(Y)$ & $f^*j^*$ & $\Omega_{\cup_f,g_1}^m(X_2)$\\ \hline
$\mathcal{D}_{\cup_f,g_2}^{i^*}(Y)$ & $\Omega^m(Y)$ & $i^*$ & $\Omega_{\cup_f,g_2}^m(X_1)$\\ \hline
$\mathcal{D}_{g_2}^{f^*j^*}(Y)$ & $\Omega^m(Y)$ & $f^*j^*$ & $\Omega_{g_2}^m(X_2)$\\ \hline
$\mathcal{D}_{f\cup g_1}^{k_1^*}(Z_1)$ & $\Omega^m(Z_1)$ & $k_1^*$ & $\Omega_{f\cup g_1}^m(X_1)$\\ \hline
$\mathcal{D}^{g_1^*l_1^*}(Z_1)$ & $\Omega^m(Z_1)$ & $g_1^*l_1^*$ & $\Omega^m(X_3)$\\ \hline
$\mathcal{D}_{g_2}^{k_2^*}(Z_2)$ & $\Omega^m(Z_2)$ & $k_2^*$ & $\Omega_{g_2}^m(X_2)$\\ \hline
$\mathcal{D}^{g_2^*l_2^*}(Z_2)$ & $\Omega^m(Z_2)$ & $g_2^*l_2^*$ & $\Omega^m(X_3)$\\ \hline
$\mathcal{D}_{\tilde{g}_1}^{\bar{k}_1^*}(\alpha_1(Z_1))$ & $\Omega^m(\alpha_1(Z_1))$ & $\bar{k}_1^*$ & $\Omega_{\tilde{g}_1}^m(X_1\cup_fX_2)$\\ \hline
$\mathcal{D}^{\tilde{g}_1^*l_1^*}(\alpha_1(Z_1))$ & $\Omega^m(\alpha_1(Z_1))$ & $\tilde{g}_1^*l_1^*$ & $\Omega^m(X_3)$\\ \hline
$\mathcal{D}_{\tilde{g}_2}^{\bar{k}_2^*}(\alpha_2(Z_2))$ & $\Omega^m(\alpha_2(Z_2))$ & $\bar{k}_2^*$ & $\Omega_{\tilde{g}_2}^m(X_1\cup_fX_2)$\\ \hline
$\mathcal{D}^{\tilde{g}_2^*l_2^*}(\alpha_2(Z_2))$ & $\Omega^m(\alpha_2(Z_2))$ & $\tilde{g}_2^*l_2^*$ & $\Omega^m(X_3)$\\ \hline
\end{tabular}
\egroup
\end{center}

For instance, the first line of the table indicates that $\mathcal{D}_{f\cup g_1}^{i^*}(Y)$ is the diffeology on $\Omega^m(Y)$ that is the pushforward of the subset diffeology on $\Omega_{f\cup g_1}^m(X_1)$ by the map $i^*$, and so on. In the above table and unless specified otherwise, the diffeology on $\Omega^m(A)$, where $A$ is a subset of a diffeological space $B$, is the standard functional diffeology on $\Omega^m(A)$ relative to the subset diffeology on $A\subseteq B$; whereas the diffeology on some $\Omega_{inv}^m(A)$ is the subset diffeology relative to the inclusion $\Omega_{inv}^m(A)\leqslant\Omega^m(A)$.

We also consider the maps $\alpha_1|_{Z_1}$ and $\alpha_2|_{Z_2}$ as maps $Z_1\to\alpha_1(Z_1)$ and $Z_2\to \alpha_2(Z_2)$ respectively, for which we have $\bar{k}_1\circ\alpha_1|_{Z_1}=\alpha_1\circ k_1$, $\bar{k}_2\circ \alpha_2|_{Z_2}=\alpha_2\circ k_2$. 

Then, the following two statements are true (recall that by definition $g_1=\tilde{g}_1\circ\alpha_1|_{Z_1}$, $g_2=\tilde{g}_2\circ \alpha_2|_{Z_2}$, and that $\alpha_1\circ k_1=\bar{k}_1\circ\alpha_1|_{Z_1}$, $\bar{k}_2\circ \alpha_2|_{Z_2}=\alpha_2|_{Z_2}\circ k_2$).

\begin{prop}\label{inherit:pshfrd:eq:prop}
We have:
\begin{enumerate}
\item If $f,g_1$ are such that $\mathcal{D}_{f\cup g_1}^{i^*}(Y)=\mathcal{D}_{\cup_f,g_1}^{f^*j^*}(Y)$ and $\mathcal{D}_{f\cup g_1}^{k_1^*}(Z_1)=\mathcal{D}^{g_1^*l_1^*}(Z_1)$, then $\mathcal{D}_{\tilde{g}_1}^{\bar{k}_1^*}(\alpha_1(Z_1))=\mathcal{D}^{\tilde{g}_1^*l_1^*}(\alpha_1(Z_1))$; 
\item If $f,g_2$ are such that $\mathcal{D}_{\cup_f,g_2}^{i^*}(Y)=\mathcal{D}_{g_2}^{f^*j^*}(Y)$ and $\mathcal{D}_{g_2}^{k_2^*}(Z_2)=\mathcal{D}^{g_2^*l_2^*}(Z_2)$, then $\mathcal{D}_{\tilde{g}_2}^{\bar{k}_2^*}(\alpha_2(Z_2))=\mathcal{D}^{\tilde{g}_2^*l_2^*}(\alpha_2(Z_2))$.
\end{enumerate}
\end{prop}

\begin{proof}
Let us prove 1. We first observe that $\mbox{Ker}(\alpha_1|_{Z_1}^*)=\{0\}$. Indeed, let $\omega\in\mbox{Ker}(\alpha_1|_{Z_1}^*)\leqslant\Omega^m(\alpha_1(Z_1))$; let $p$ be a plot of $\alpha_1(Z_1)$. Since $f$ is a subduction, $p$ locally has form $\alpha_1|_{Z_1}\circ p_1$ for some plot $p_1$ of $Z_1$. Therefore $\omega(p)=\alpha_1|_{Z_1}^*\omega(p_1)=0$, which means that $\omega$ is the zero form.

Let now $q_{1,2}:U\to\Omega_{\tilde{g}_1}^m(X_1\cup_fX_2)$ be a plot, and let $(q_1,q_2)=\pi^*q$. Then by Theorem~\ref{f,g-inv:iterat:thm} and up to appropriately shrinking $U$ we can assume that $q_1$ is a plot of $\Omega_{f\cup g_1}^m(X_1)$, $q_2$ is a plot of $\Omega_{\cup_f,g_1}^m(X_2)$, and by the second equality in the assumption there exists a plot $q_3$ of $\Omega^m(X_3)$ such that $k_1^*q_1=g_1^*l_1^*q_3=\alpha_1^*\tilde{g}_1^*l_1^*q_3$. Now, since $q_1=\alpha_1^*q_{1,2}$, by definition of $\bar{k}_1$ we obtain $\alpha_1|_{Z_1}^*\bar{k}_1^*q_{1,2}=\alpha_1|_{Z_1}^*\tilde{g}_1^*l_1^*q_3$, and since $\mbox{Ker}(\alpha_1|_{Z_1}^*)=\{0\}$, we obtain $\bar{k}_1^*q_{1,2}=\tilde{g}_1^*l_1^*q_3$, so that $\mathcal{D}_{\tilde{g}_1}^{\bar{k}_1^*}(\alpha_1(Z_1))\subseteq\mathcal{D}^{\tilde{g}_1^*l_1^*}(\alpha_1(Z_1))$.

\emph{Vice versa}, let $q_3:U\to\Omega^m(X_3)$ be a plot. By the second and the first part of the assumption respectively, and up to appropriately shrinking $U$, there exist a plot $q_1$ of $\Omega_{f\cup g_1}^m(X_1)$ such that $k_1^*q_1=g_1^*l_1^*q_3$ and a plot $q_2$ of $\Omega_{\cup f,g_1}^m(X_2)$ such that $k_1^*q_1=f^*j^*q_2$. It now follows from Theorems~\ref{omega:of:glued:inverse:smooth:thm} and \ref{f,g-inv:iterat:thm} that there exists a plot $q_{1,2}$ of $\Omega_{\tilde{g}_1}^m(X_1\cup_fX_2)$ such that $\pi^*q_{1,2}=(q_1,q_2)$, and we have, as before, $k_1^*q_1=k_1^*\alpha_1^*q_{1,2}=\alpha_1^*\bar{k}_1^*q_{1,2}$, $g_1^*l_1^*q_3=\alpha_1^*\tilde{g}_1^*l_1^*q_3$. Thus, $\alpha_1^*\bar{k}_1^*q_{1,2}=\alpha_1^*\tilde{g}_1^*l_1^*q_3$, and again, since $\mbox{Ker}(\alpha_1|_{Z_1}^*)=\{0\}$, we have $\bar{k}_1^*q_{1,2}=\tilde{g}_1^*l_1^*q_3$, so that $\mathcal{D}_{\tilde{g}_1}^{\bar{k}_1^*}(\alpha_1(Z_1))\supseteq\mathcal{D}^{\tilde{g}_1^*l_1^*}(\alpha_1(Z_1))$.

Let us now prove 2. Observe first of all that, since $f$ is a subduction, $\alpha_2|_{Z_2}$ is a diffeomorphism, so we immediately obtain that $\mbox{Ker}(\alpha_2|_{Z_2}^*)=\{0\}$. Let first $q_3:U\to\Omega^m(X_3)$ be a plot; by the second and first equalities in the assumption, and up to apropriately shrinking $U$, there exist a plot $q_2$ of $\Omega_{g_2}^m(X_2)$ such $g_2^*l_2^*q_3=k_2^*q_2$, and a plot $q_1$ of $\Omega_{\cup_f,g_2}^m(X_1)$ such that $f^*j^*q_2=i^*q_1$. It then follows that $(q_1,q_2)$ is a plot of $\Omega_{\cup_f,g_2}^m(X_1)\oplus_{comp}\Omega_{g_2}^m(X_2)$. Define $q_{1,2}:=\mathcal{L}\circ(q_1,q_2)$, where $\mathcal{L}$ is the diffeomorphism of Theorem~\ref{omega:of:glued:inverse:smooth:thm}. Then by Theorem~\ref{f,g-inv:iterat:thm} $q_{1,2}$ is a plot of $\Omega_{\tilde{g}_2}^m(X_1\cup_fX_2)$, and since $\alpha_2^*q_{1,2}=q_2$ by construction, it satisfies the equality $g_2^*l_2^*q_3=k_2^*\alpha_2^*q_{1,2}$, hence the equality $\alpha_2|_{Z_2}^*\tilde{g}_2^*l_2^*q_3=\alpha_2|_{Z_2}^*\bar{k}_2^*q_{1,2}$. Finally, since $\alpha_2|_{Z_2}^*$ is injective, we obtain $\tilde{g}_2^*l_2^*q_3=\bar{k}_2^*q_{1,2}$, as wanted.

\emph{Vice versa}, let $q_{1,2}:U\to\Omega_{\tilde{g}_2}^m(X_1\cup_fX_2)$ be a plot, and let $(q_1,q_2)=\pi^*q_{1,2}$ (recall that $\pi:X_1\sqcup X_2\to X_1\cup_fX_2$ is the natural projection). By assumption, for U small enough there exists a plot $q_3$ of $\Omega^m(X_3)$ such that $g_2^*l_2^*q_3=k_2^*q_2$. Then, since $q_2=\alpha_2^*q_{1,2}$, we have $\alpha_2|_{Z_2}^*\tilde{g}_2^*l_2^*q_3=\alpha_2|_{Z_2}^*\bar{k}_2^*q_{1,2}$, and again, $\alpha_2|_{Z_2}^*$ being injective, we conclude that $\tilde{g}_2^*l_2^*q_3=\bar{k}_2^*q_{1,2}$.
\end{proof}

\begin{example}\label{phfrd:dfgies:expl}
Let us consider Example 4 and $m=0$ (the case $m=1$ is similar). The meaning of the equality $\mathcal{D}_{f\cup g_1}^{i^*}(Y)=\mathcal{D}_{\cup_f,g_1}^{f^*j^*}(Y)$ is that for any smooth map $P_1:U\times\mathbb{R}^2\to\mathbb{R}$ such that $P_1(u,t,0)=P_1(u,0,t)\,\,\forall u\mbox{ and }\forall t>0$ and $P_1(u,t,0)=P_1(u,\pi-t,0)\,\,\forall u\mbox{ and }\forall t\in(0,\pi)$ there must exist a smooth map $P_2:U\times\mathbb{R}\to\mathbb{R}$ such that $P_2(u,t)=P_2(u,\pi-t)=P_1(u,t,0)\,\,\forall u\mbox{ and }\forall t\in(0,\pi)$, and \emph{vice versa}. This is easily seen to be the case: given $P_1$, the desired $P_2$ can be obtained by setting $P_2(u,t)=P_1(u,t,0)$ everywhere, while given $P_2$, the desired $P_1$ can be defined by $P_1(u,x_0,x_1)=P_2(u,x_0+x_1)$. 

The meaning of the equality $\mathcal{D}_{f\cup g_1}^{k_1^*}(Z_1)=\mathcal{D}^{g_1^*l_1^*}(Z_1)$ is that for any smooth map $P_1:U\times\mathbb{R}^2\to\mathbb{R}$ such that $P_1(u,t,0)=P_1(u,0,t)\,\,\forall u\mbox{ and }\forall t>0$ and $P_1(u,t,0)=P_1(u,\pi-t,0)\,\,\forall u\mbox{ and }\forall t\in(0,\pi)$ there must exist a smooth map $P_3:U\times\mathbb{R}\to\mathbb{R}$ such that $P_1(u,t,0)=P_3(u,\sin t)\,\,\forall u\mbox{ and }\forall t\in(0,\pi)$, and \emph{vice versa}. The inclusion $\mathcal{D}_{f\cup g_1}^{k_1^*}(Z_1)\supseteq\mathcal{D}^{g_1^*l_1^*}(Z_1)$ is trivial: given $P_3$, we obtain $P_1$ by setting $P_1(u,x_0,x_1)=P_3(u,\sin(x_0+x_1))$ for all $u$ and for all $t\in(0,\pi)$. On the other hand, the reverse inclusion is potentially more complicated; it is easy to see that it is equivalent to there being, for every smooth $g:U\times\mathbb{R}\to\mathbb{R}$ such that $g(u,t)=g(u,\pi-t)\,\forall u\,\forall t\in(0,\pi)$, a smooth, on $U\times\mathbb{R}$, $P_3$ such that $g(u,t)=P_3(u,\sin t)$ (for any fixed $u_0$ this indeed appears to follow from considering the Fourier series of the even function $h:[-\frac{\pi}{2},\frac{\pi}{2}]\to\mathbb{R}$ given by $h(t)=g(u_0,t+\frac{\pi}{2})$)\footnote{Since $h$ is smooth and even, we have $h(t)=a_0+\sum_{n=1}^\infty a_n\cos(2nt)$, where $a_n$'s form a rapidly decaying sequence. By direct calculation $g(u_0,t)=a_0+\sum_{n=1}^\infty a_nQ_n(\sin t)$, where $Q_n(s)$ is a polynomial of degree at most $4n$ and such that $|Q_n(s)|\leqslant 10n$ for all $s\in[0,1]$; thus, we also have $|Q_n^{(m)}(s)|\leqslant 4^m\cdot 10\cdot n^{m+1}$. Since the sequence $\{a_n\}$ decays rapidly, by Weierstrass M-test the series $a_0+\sum_{n=1}^\infty a_nQ_n(s)$ converges uniformly on $[0,1]$ to some $h_3$; observe that on $(0,1)$ $h_3$ is differentiable as a uniform limit of differentiable functions, and by construction it is bounded on $[0,1]$. Furthermore, by all the same arguments all series $a_0+\sum_{n=1}^\infty a_nQ_n^{(m)}(s)$ converge uniformly to functions differentiable on $(0,1)$ and bounded on $[0,1]$ (and more precisely, on $(0,1)$ they converge to the derivatives of $h_3$). It then follows that $h_3$ has all the continuous derivatives at $0$ and at $1$ and therefore by Seeley's extension theorem \cite{seeley} it can be extended to a smooth function on $\mathbb{R}$. By varying $u_0$ in $U$, we thus obtain a function $P_3(u,t)$ on $U\times\mathbb{R}$ that satisfies $P_1(u,t,0)=P_3(u,\sin t)$ for all relevant $u,t$ and is smooth in it for every fixed $u$, but to ensure that it is smooth as a whole, we essentially need to ensure that the Seeley's operator is smooth in the functional diffeology}.

The meaning of the equality $\mathcal{D}_{\cup_f,g_2}^{i^*}(Y)=\mathcal{D}_{g_2}^{f^*j^*}(Y)$ is that for any smooth map $P_1:U\times\mathbb{R}^2\to\mathbb{R}$ such that $P_1(u,t,0)=P_1(u,0,t)\,\forall u\,\forall t>0$ and $P_1(u,t,0)=P_1(u,2-t,0)\,\forall u\,\forall t\in(0,2)$ there must exist a smooth $P_2:U\times\mathbb{R}\to\mathbb{R}$ such that $P_2(u,t)=P_2(u,2-t)\,\forall u\,\forall t\in(0,2)$ and $P_1(u,t,0)=P_2(u,t)\,\forall u\,\forall t>0$, and \emph{vice versa}. This is analogous to that of the condition $\mathcal{D}_{f\cup g_1}^{i^*}(Y)=\mathcal{D}_{\cup_f,g_1}^{f^*j^*}(Y)$: given $P_1$, we obtain the desired $P_2$ by setting $P_2(u,t)=P_1(u,t,0)$, while given $P_2$, the desired $P_1$ is defined by $P_1(u,x_0,x_1)=P_2(u,x_0+x_1)$.

The meaning of the equality $\mathcal{D}_{g_2}^{k_2^*}(Z_2)=\mathcal{D}^{g_2^*l_2^*}(Z_2)$ is there being, for any smooth map $P_2:U\times\mathbb{R}\to\mathbb{R}$ such that $P_2(u,t)=P_2(u,2-t)\,\forall u\,\forall t\in(0,2)$, a smooth map $P_3:U\times\mathbb{R}\to\mathbb{R}$ such that $P_2(u,t)=P_3(u,(t-1)^2)\,\forall u\,\forall t\in(0,2)$, and \emph{vice versa}. Also in this case the inclusion $\mathcal{D}_{g_2}^{k_2^*}(Z_2)\supseteq\mathcal{D}^{g_2^*l_2^*}(Z_2)$ is trivial: given $P_3$, we obtain $P_2$ by setting $P_2(u,t)=P_3(u,(t-1)^2)$. On the other hand, the reverse inclusion is less trivial: it amounts to the existence, for every smooth $P_2$ such that $P_2(u,t)=P_2(u,2-t,0)\,\forall u\,\forall t\in(0,2)$, of a smooth $P_3$ such that $P_2(u,t)=P_3(u,(t-1)^2)\,\forall u\,\forall t\in(0,2)$. We observe that for a fixed $u$ at least (probably not only), the existence of such $P_3$ follows from Theorem 1 (see also Remark p.160) of \cite{whit43} on even smooth functions applied to the function $h:(-1,1)\ni t\mapsto P_2(u,t+1)$. 
\end{example}

\begin{example}
In the case of Example~\ref{example5:ex} it is easy to see that the conditions $\mathcal{D}_{f\cup g_1}^{i^*}(Y)=\mathcal{D}_{\cup_f,g_1}^{f^*j^*}(Y)$ and $\mathcal{D}_{f\cup g_1}^{k_1^*}(Z_1)=\mathcal{D}^{g_1^*l_1^*}(Z_1)$ are satisfied for all $m$, so the result of gluing inherits the property $\mathcal{D}_{\tilde{g}_1}^{\bar{k}_1^*}(\alpha_1(Z_1))=\mathcal{D}^{\tilde{g}_1^*l_1^*}(\alpha_1(Z_1))$.
\end{example}

\subsubsection{Gluing along diffeomorphisms}\label{case:diffeo:sect} 

Another simplification of invariance conditions can be obtained in the case when at least one of the maps $f,g_1,g_2$ involved in gluing, is a diffeomorphism.

\paragraph{When only one diffeomorphism is involved: equivalent plots and invariant forms} Let us first provide some extra comments in the case when just one of the maps $f,g_1,g_2$ is a diffeomorphism.

\emph{Suppose that $f$ is a diffeomorphism.} In this case it is easy to see that $f\cup g_1$-equivalent plots are precisely $g_1$-equivalent plots. On the other hand, while the definitions of $(\cup_f,g_1)$-equivalent plots and of $(\cup_f,g_2)$-equivalent plots can be reformulated somewhat more concisely: $p_2(u)\neq p_2'(u)$ $\Rightarrow$ $p_2(u),p_2'(u)\in f(Y\cap Z_2)$ $\&$ $g_1(f^{-1}(p_2(u)))=g_1(f^{-1}(p_2'(u)))$ and $p_1(u)\neq p_1'(u)$ $\Rightarrow$ $p_1(u),p_1'(u)\in f^{-1}(Z_1)$ $\&$ $g_2(f(p_1(u)))=g_2(f(p_1'(u)))$, this reformulation does not lead to any evident conclusion. Likewise, the only substantial simplification that we obtain for invariant forms is the equality $\Omega_{f\cup g_1}^m(X_1)=\Omega_{g_1}^m(X_1)$, and so $\Omega_{\tilde{g}_1}^m(X_1\cup_fX_2)\cong\Omega_{g_1}^m(X_1)\oplus_{comp}\Omega_{\cup_f,g_1}^m(X_2)$ (while the expression for $\Omega_{\tilde{g}_2}^m(X_1\cup_fX_2)$ cannot be simplified).

\emph{Suppose that $g_1$ is a diffeomorphism.} It is easy to see that in this case $f\cup g_1$-equivalent plots are precisely the $f$-equivalent plots, while any plot of $X_2$ is $(\cup_f,g_1)$-equivalent only to itself. Accordingly, we have $\Omega_{f\cup g_1}^m(X_1)=\Omega_f^m(X_1)$, $\Omega_{\cup_f,g_1}^m(X_2)=\Omega^m(X_2)$; in particular, $\Omega_{\tilde{g}_1}^m(X_1\cup_fX_2)\cong\Omega_f^m(X_1)\oplus\Omega^m(X_2)\cong\Omega^m(X_1\cup_fX_2)$.

\emph{Suppose that $g_2$ is a diffeomorphism.} Then again $(\cup_f,g_2)$-equivalent plots of $X_1$ are precisely the $f$-equivalent plots, while obviously any plot of $X_2$ is $g_2$-equivalent only to itself. Accordingly, $\Omega_{\cup_f,g_1}^m(X_1)=\Omega_f^m(X_1)$, $\Omega_{g_2}^m(X_2)=\Omega^m(X_2)$; in particular, $\Omega_{\tilde{g}_2}^m(X_1\cup_fX_2)\cong\Omega_f^m(X_1)\oplus\Omega^m(X_2)\cong\Omega^m(X_1\cup_fX_2)$.

\paragraph{When all gluing maps are diffeomorphisms} In this case all forms in $\Omega^m(X_1)$, $\Omega^m(X_2)$ satisfy every applicable condition of invariance, so we have $\Omega_{\tilde{g}_1}^m(X_1\cup_fX_2)=\Omega_{\tilde{g}_2}^m(X_1\cup_fX_2)=\Omega^m(X_1\cup_fX_2)$. We furthermore observe that there is the following commutativity property.

\begin{lemma}\label{commut:glu:lem}
If $f$ is a diffeomorphism of its domain with its image then there is a diffeomorphism $$X_1\cup_fX_2\cong X_2\cup_{f^{-1}}X_1.$$
\end{lemma}

\begin{proof}
The desired diffeomorphism is (of course) the pushforward of the obvious diffeomorphism $X_1\sqcup X_2\cong X_2\sqcup X_1$ by the defining projections $X_1\sqcup X_2\to X_1\cup_fX_2$ and $X_2\sqcup X_1\to X_2\cup_{f^{-1}}X_1$.
\end{proof}

Observe in particular that thanks to this lemma and Theorem~\ref{gluing:assctv:thm}, in the case of gluings along diffeomorphisms only, we can consider only the space $(X_1\cup_fX_2)\cup_{\tilde{g}}X_3$ as a model of iterated gluing, since in this case Lemma~\ref{commut:glu:lem} and Corollary~\ref{omega:iterat:cor} allow to calculate the spaces of differential forms for any possible configuration of gluings of three diffeological spaces.

\subsubsection{Other configurations of gluing}\label{other:glue:config:sect}

We have so far considered one particular type of gluing of three diffeological spaces, namely one where first two diffeological spaces are glued to each other, and then the result is glued to some third space. As has just been observed at the end of the previous section, this is sufficient in the case when gluing is performed along diffeomorphisms, but, as has been noted in Section~\ref{notation:iter:sect}, in general there are other possible configurations of gluings which \emph{a priori} should be considered separately; namely, in certain cases the two gluings can be performed in a different order, and moreover, with a slightly different setup, after gluing together two spaces, we can glue some third space \emph{to} the result. In this concluding section we consider these possibilities.

\paragraph{Changing the order of gluing} This involves considering the spaces $X_1\cup_{g_1}X_3$ and  $X_2\cup_{g_2}X_3$, and then using the map $f$, respectively, either to glue $X_1\cup_{g_1}X_3$ to $X_2$, or to glue $X_1$ to $X_2\cup_{g_2}X_3$. This requires an additional assumption, namely, we also assume that $g_1,g_2$ are subductions (as is $f$) and $g_1$ is such that $g_1(z_1)=g_1(z_1')\Rightarrow f(z_1)=f(z_1')$ for all $z_1,z_1'\in Y\cap Z_1$ (so that $f(z)=f(z')\Leftrightarrow g_1(z_1)=g_1(z_1')$ for all $z_1,z_1'\in Y\cap Z_1$). To describe the purported gluings, we introduce, for the spaces $X_1\cup_{g_1}X_3$ and  $X_2\cup_{g_2}X_3$, the analogues of maps $\alpha_1$ and $\alpha_2$, which we denote $$\beta_1:X_1\to X_1\cup_{g_1}X_3,\,\,\,\beta_3:X_3\to X_1\cup_{g_1}X_3,\,\,\,\gamma_2:X_2\to X_2\cup_{g_2}X_3,\,\,\,\gamma_3:X_3\to X_2\cup_{g_2}X_3,$$ and we define $\tilde{f}_1:\beta_1(Y)\to X_2$ to be such that $f=\tilde{f}_1\circ\beta_1|_Y$ and $\tilde{f}_2:Y\to X_2\cup_{g_2}X_3$ to be $\tilde{f}_2=\gamma_2\circ f$. We then have the following statement.

\begin{thm}\label{gluing:assctv:thm}
Let $f,g_1,g_2$ be as above. Then there are diffeomorphisms: $$(X_1\cup_fX_2)\cup_{\tilde{g}_1}X_3\cong (X_1\cup_{g_1}X_3)\cup_{\tilde{f}_1}X_2$$ and $$(X_1\cup_fX_2)\cup_{\tilde{g}_2}X_3\cong X_1\cup_{\tilde{f}_2}(X_2\cup_{g_2}X_3).$$
\end{thm}

\begin{proof}
It is obvious that there is a commutativity-associativity diffeomorphism $$(X_1\sqcup X_2)\sqcup X_3\cong (X_1\sqcup X_3)\sqcup X_2$$ and an associativity diffeomorphism $$(X_1\sqcup X_2)\sqcup X_3\cong X_1\sqcup (X_2\sqcup X_3).$$ Denote by 
\[\begin{array}{lcr}
\pi_f:X_1\sqcup X_2\to X_1\cup_fX_2, & \pi_{g_1}:X_1\sqcup X_3\to X_1\cup_{g_1}X_3, & \pi_{g_2}:X_2\sqcup X_3\to X_2\cup_gX_3,\\ 
\end{array}\]
\[\begin{array}{ll}
\pi_{\tilde{g}_1}:(X_1\cup_fX_2)\sqcup X_3\to (X_1\cup_fX_2)\cup_{\tilde{g}_1}X_3,& \pi_{\tilde{g}_2}:(X_1\cup_fX_2)\sqcup X_3\to (X_1\cup_fX_2)\cup_{\tilde{g}_2}X_3, \\
\pi_{\tilde{f}_1}:(X_1\cup_{g_1}X_3)\sqcup X_2\to(X_1\cup_{g_1}X_3)\cup_{\tilde{f}_1}X_2, & \pi_{\tilde{f}_2}:X_1\sqcup(X_2\cup_{g_2}X_3)\to X_1\cup_{\tilde{f}_2}(X_2\cup_{g_2}X_3) 
\end{array}\]
the corresponding natural projections, and by 
\[\begin{array}{lcr}
\mbox{id}_{X_1}:X_1\to X_1, & \mbox{id}_{X_2}:X_2\to X_2, & \mbox{id}_{X_3}:X_3\to X_3
\end{array}\] the identity maps; the maps 
\[\begin{array}{lr}
\pi_f\sqcup\mbox{id}_{X_3}:(X_1\sqcup X_2)\sqcup X_3\to(X_1\cup_fX_2)\sqcup X_3, &
\pi_{g_1}\sqcup\mbox{id}_{X_2}:(X_1\sqcup X_3)\sqcup X_2\to(X_1\cup_{g_1}X_3)\sqcup X_2, 
\end{array}\]
$$\mbox{id}_{X_1}\sqcup\pi_{g_2}:X_1\sqcup(X_2\sqcup X_3)\to X_1\sqcup(X_2\cup_gX_3)$$ 
are then defined in the obvious way.

Let us prove the existence of the first diffeomorphism. Since $f$ and $g_1$ are subductions, it suffices to prove that there is a well-defined pushforward of the above-mentioned commutativity-associativity diffeomorphism by the maps $\pi_{\tilde{g}_1}\circ(\pi_f\sqcup\mbox{id}_{X_3})$ and $\pi_{\tilde{f}_1}\circ(\pi_{g_1}\sqcup\mbox{id}_{X_2})$.

Let $x_1,x_1'\in X_1$, $x_2,x_2'\in X_2$, $x_3,x_3'\in X_3$. To these correspond well-defined elements $[x_i]_1=(\pi_{\tilde{g}_1}\circ(\pi_f\sqcup\mbox{id}_{X_3}))(x_i)$ or $[x_i']_1=(\pi_{\tilde{g}_1}\circ(\pi_f\sqcup\mbox{id}_{X_3}))(x_i')$, $i=1,2,3$, of $(X_1\cup_fX_2)\cup_{\tilde{g}_1}X_3$, as well as well-defined elements $[x_i]_2=(\pi_{\tilde{f}_1}\circ(\pi_{g_1}\sqcup\mbox{id}_{X_2}))(x_i)$ or $[x_i']_2=(\pi_{\tilde{f}_1}\circ(\pi_{g_1}\sqcup\mbox{id}_{X_2}))(x_i')$, $i=1,2,3$, of $(X_1\cup_{g_1}X_3)\cup_{\tilde{f}_1}X_2$.\footnote{In these formulae by $x_i$ or $x_i'$ we actually mean its obvious image in the appropriate disjoint union.} It is sufficent to prove, for all $i$, that $[x_i]_1=[x_i']_1$ $\Leftrightarrow$ $[x_i]_2=[x_i']_2$.

Let $i=1$. Then $[x_i]_1=[x_i']_1$ if and only if $x_1,x_1'\in Y\cup Z_1$ and the following are true:
\begin{enumerate}
\item If $x_1,x_1'\in Y$ then either $f(x_1)=f(x_1')$ or there exist $\bar{x}_1,\bar{x}_1'\in Y\cap Z_1$ such that $f(\bar{x}_1)=f(x_1)$, $f(\bar{x}_1')=f(x_1')$, and $g_1(\bar{x}_1)=g_1(\bar{x}_1')$;
\item If $x_1,x_1'\in Z_1$ then either $g_1(x_1)=g_1(x_1')$ or there exist $\bar{x}_1,\bar{x}_1'\in Y\cap Z_1$ such that $g_1(\bar{x}_1)=g_1(x_1)$, $g_1(\bar{x}_1')=g_1(x_1')$, and $f(\bar{x}_1)=f(\bar{x}_1')$;
\item If $x_1\in Y$ and $x_1'\in Z_1$ (or $x_1\in Z_1$ and $x_1'\in Y$) then there exists $\bar{x}_1\in Y\cap Z_1$ (or resp. $\bar{x}_1'\in Y\cap Z_1$) such that $f(\bar{x}_1)=f(x_1)$ and $g_1(\bar{x}_1)=g_1(x_1')$ (resp. $f(\bar{x}_1')=f(x_1')$ and $g_1(\bar{x}_1')=g_1(x_1)$).
\end{enumerate}
Whereas, $[x_i]_2=[x_i']_2$ if and only if $x_1,x_1'\in Y\cup Z_1$ and the following are true:
\begin{enumerate}
\item[1a] If $x_1,x_1'\in Y$ then either $f(x_1)=f(x_1')$ or there exist $\bar{x}_1,\bar{x}_1'\in Y\cap Z_1$ such that $g_1(\bar{x}_1)=g_1(x_1)$, $g_1(\bar{x}_1')=g_1(x_1')$, and $f(\bar{x}_1)=f(\bar{x}_1')$;
\item[2a] If $x_1,x_1'\in Z_1$ then either $g_1(x_1)=g_1(x_1')$ or there exist $\bar{x}_1,\bar{x}_1'\in Y\cap Z_1$ such that $g_1(\bar{x}_1)=g_1(x_1)$, $g_1(\bar{x}_1')=g_1(x_1')$, and $f(\bar{x}_1)=f(\bar{x}_1')$;
\item[3a] If $x_1\in Y$ and $x_1'\in Z_1$ (or $x_1\in Z_1$ and $x_1'\in Y$) then there exists $\bar{x}_1\in Y\cap Z_1$ (or resp. $\bar{x}_1'\in Y\cap Z_1$) such that $f(\bar{x}_1)=f(x_1)$ and $g_1(\bar{x}_1)=g_1(x_1')$ (resp. $f(\bar{x}_1')=f(x_1')$ and $g_1(\bar{x}_1')=g_1(x_1)$).
\end{enumerate}
We can now observe that 2 and 2a, as well as 3 and 3a, coincide word for word, while 1 and 1a are equivalent by the assumption on $g_1$.

Let now $i=2$. Then $[x_2]_1=[x_2']_1$ if and only if $x_2,x_2'\in f(Y)$, and there exist $x_1\in f^{-1}(x_2)\cap Z_1$, $x_1'\in f^{-1}(x_2')\cap Z_1$ such that $g_1(x_1)=g_1(x_1')$, while $[x_i]_2=[x_i']_2$ if and only if $x_2=x_2'$. Now, $g_1(x_1)=g_1(x_1')$ implies $f(x_1)=f(x_1')$ by assumption on $g_1$, and this implies $x_2=x_2'$ by choice of $x_1,x_1'$. Therefore  $[x_2]_1=[x_2']_1$ $\Leftrightarrow$ $[x_2]_2=[x_1']_2$.

Finally, let $i=3$. Then $[x_3]_1=[x_3']_1$ if and only if $x_3=x_3'$, while $[x_3]_2=[x_3']_2$ if and only if $x_3,x_3'\in g_1(Z_1)$ and there exist $x_1\in g_1^{-1}(x_3)\cap Y$, $x_1'\in g_1^{-1}(x_3')\cap Y$, and the two conditions are equivalent by the same reasoning as in the case $i=2$. Thus, the existence of the first diffeomorphism has been proven.

Let us now prove the existence of the second diffeomorphism. Again, $f$ and $g_2$ being subductions, it suffices to prove that there is a well-defined pushforward of the above associativity diffeomorphism by the maps $\pi_{\tilde{g}_2}\circ(\pi_f\sqcup\mbox{id}_{X_3})$ and $\pi_{\tilde{f}_2}\circ(\mbox{id}_{X_1}\sqcup\pi_{g_2})$. Let again $x_i,x_i'\in X_i$, $i=1,2,3$, and let $[x_i]_1=(\pi_{\tilde{g}_2}\circ(\pi_f\sqcup\mbox{id}_{X_3}))(x_i)$, $[x_i']_1=(\pi_{\tilde{g}_2}\circ(\pi_f\sqcup\mbox{id}_{X_3}))(x_i')$ for $i=1,2,3$, and let $[x_i]_2=(\pi_{\tilde{f}_2}\circ(\mbox{id}_{X_1}\sqcup\pi_{g_2}))(x_i)$, $[x_i']_2=(\pi_{\tilde{f}_2}\circ(\mbox{id}_{X_1}\sqcup\pi_{g_2}))(x_i')$, again for $i=1,2,3$. Let us prove that $[x_i]_1=[x_i']_1$ $\Leftrightarrow$ $[x_i]_2=[x_i']_2$. 

Let first $i=1$. Then $[x_1]_1=[x_1']_1$ if and only if $x_1,x_1'\in Y$ and either $f(x_1)=f(x_1')$ and $f(x_1),f(x_1')\in Z_2$ and $g_2(f(x_1))=g_2(f(x_1'))$, and the equality $[x_1]_2=[x_1']_2$ corresponds to exactly the same condition. Likewise, in the case of $i=2$ both $[x_2]_1=[x_2']_1$ and $[x_2]_2=[x_2']_2$ are equivalent to the condition $x_2,x_2'\in Z_2$ and $g_2(x_2)=g_2(x_2')$. Finally, in case $i=3$ both $[x_3]_1=[x_3']_1$ and $[x_3]_2=[x_3']_2$ are equivalent to $x_3=x_3'$, hence they themselves are equivalent. All cases having been exhausted, the theorem is proven.
\end{proof}

We can also observe the following.

\begin{cor}\label{case:assoc:omega:cor}
The corresponding diffeomorphisms $$\Omega^m((X_1\cup_fX_2)\cup_{\tilde{g}_1}X_3)\cong\Omega^m((X_1\cup_{g_1}X_3)\cup_{\tilde{f}_1}X_2),$$ $$\Omega^m((X_1\cup_fX_2)\cup_{\tilde{g}_2}X_3)\cong\Omega^m(X_1\cup_{\tilde{f}_2}(X_2\cup_{g_2}X_3))$$ are given, respectively, by $$(\omega_1\cup_f\omega_2)\cup_{\tilde{g}_1}\omega_3\mapsto(\omega_1\cup_{g_1}\omega_3)\cup_{\tilde{f}_1}\omega_2,\,\,\,(\omega_1\cup_f\omega_2)\cup_{\tilde{g}_2}\omega_3\mapsto\omega_1\cup_{\tilde{f}_2}(\omega_2\cup_{g_2}\omega_3).$$
\end{cor}

\paragraph{The remaining case} The final case to consider is the one of the second gluing map being defined on $X_3$ with values in $X_i$, which is then glued to $X_1\cup_fX_2$. The notation for this case is as follows. Let $X_1,X_2,X_3$ be diffeological spaces, let $Y\subseteq X_1$, $W_1,W_2\subseteq X_3$, let $f:Y\to X_2$ be a subduction, and let $h_i:W_i\to X_i$, $i=1,2$, be smooth. Then we can form the spaces $X_3\cup_{\bar{h}_i}(X_1\cup_fX_2))$, where $\bar{h}_i=\alpha_i\circ h_i$ (with $\alpha_i:X_i\to X_1\cup_fX_2$ as before). Then by Corollary~\ref{omega:of:glued:cor} $\Omega^m(X_3\cup_{\bar{h}_i}(X_1\cup_fX_2))\cong\Omega_{\bar{h}_i}^m(X_3)\oplus_{comp}\Omega^m(X_1\cup_fX_2)$ with the compatibility being with respect to the map $\bar{h}_i$. 

Observe that, since $\alpha_2$ is injective, we have $\Omega_{\bar{h}_2}^m(X_3)=\Omega_{h_2}^m(X_3)$, whereas we need the following description of $\Omega_{\bar{h}_1}^m(X_3)$.

\begin{defn}
Two plots $p_3,p_3':U\to X_3$ are said to be \textbf{$(h_1,f)$-equivalent} if for all $u$ such that $p_3(u)\neq p_3'(u)$ we have $p_3(u),p_3'(u)\in W_1$ and either $h_1(p_3(u))=h_1(p_3'(u))$ or $h_1(p_3(u)),h_1(p_3'(u))\in Y$ and $f(h_1(p_3(u)))=f(h_1(p_3'(u)))$. A form $\omega_3\in\Omega^m(X_3)$ is said to be \textbf{$(h_1,f)$-invariant} if for every two $(h_1,f)$-equivalent plots $p_3,p_3'$ we have $\omega_3(p_3)=\omega_3(p_3')$.
\end{defn}

The space of all $(h_1,f)$-invariant $m$-forms on $X_3$ will be denoted by $\Omega_{h_1,f}^m(X_3)$.

\begin{example}
Let $X_1,X_2,X_3$, and $f$ be as in Example 4, and let $h_1$, defined on $(0,\pi)\subset X_3$ act by $h_1(t)=(\sin t,0)$. Since $f$ is injective on the image of $h_1$, we immediately observe that the $(h_1,f)$-equivalent plots are precisely $h_1$-equivalent plots, and we have $\Omega_{h_1,f}^0(X_3)=\{r_3|r_3(t)=r_3(\pi-t)\,\forall t\in(0,\pi)\}$, $\Omega_{h_1,f}^1(X_3)=\{r_3dt|r_3(t)=-r_3(\pi-t)\,\forall t\in(0,\pi)\}$, where in both cases $r_3$ is a smooth function $(0,+\infty)\to\mathbb{R}$ that admits a smooth extension to the entire $\mathbb{R}$.
\end{example}

We now obtain the following statement.

\begin{prop}
We have $\Omega_{\bar{h}_1}^m(X_3)=\Omega_{h_1,f}^m(X_3)$.
\end{prop}

\begin{proof}
This is a direct consequence of the definition of $(h_1,f)$-invariant forms.
\end{proof}

We also need the following analogue of Proposition~\ref{comp:iter:prop}.

\begin{prop}
Let $\omega_i\in\Omega^m(X_i)$, $i=1,2,3$, be such that $\omega_1$ and $\omega_2$ are $f$-compatible. Then the forms $\omega_3$ and $\omega_1\cup_f\omega_2$ are $\bar{h}_i$-compatible if and only if $\omega_3$ and $\omega_i$ are $h_i$-compatible.
\end{prop}

\begin{proof}
Denote by $k_i:W_i\hookrightarrow X_3$, $l_i:h_i(W_i)\hookrightarrow X_i$, $\bar{l}_i:\bar{h}_i(W_i)\hookrightarrow X_1\cup_fX_2$ the corresponding natural inclusions. The proof follows from the following commutative diagram
$$
\begin{CD}
U @>h_ip_i>> h_i(W_i) @>l_i>> X_i\\
@\vert @. @V\alpha_iVV \\
U @>\bar{h}_ip_i>> \bar{h}_i(W_i) @>\bar{l}_i>> X_1\cup_fX_2
\end{CD}
$$
where $i=1,2$ and $p_i:U\to W_i\subseteq X_3$ is any plot of the subset diffeology on $W_i$, and from the fact that $\alpha_i^*(\omega_1\cup_f\omega_2)=\omega_i$.
\end{proof}

We thus obtain the following conclusion (recall that $i:Y\hookrightarrow X_1$, $j:f(Y)\hookrightarrow X_2$ are the natural inclusions).

\begin{cor}
The space $\Omega^m(X_3\cup_{\bar{h}_1}(X_1\cup_fX_2))$ is diffeomorphic to $$\{(\omega_3,\omega_1,\omega_2)\,|\,\omega_3\in\Omega_{h_1,f}^m(X_3),\omega_1\in\Omega_f^m(X_1),\omega_2\in\Omega^m(X_2),k_1^*\omega_3=h_1^*l_1^*\omega_1,\,i^*\omega_1=f^*j^*\omega_2\},$$ and the space $\Omega^m(X_3\cup_{\bar{h}_i}(X_1\cup_fX_2))$ is diffeomorphic to $$\{(\omega_3,\omega_1,\omega_2)\,|\,\omega_3\in\Omega_{h_2}^m(X_3),\omega_1\in\Omega_f^m(X_1),\omega_2\in\Omega^m(X_2),k_2^*\omega_3=h_2^*l_2^*\omega_2,\,i^*\omega_1=f^*j^*\omega_2\}.$$
\end{cor}

\begin{example}
Observe that the $0$-forms $r_3$ and $r_1$ on $X_3$ and $X_1$ respectively are $h_1$-compatible iff $r_3(t)=r_1(\sin t,0)$ for all $t\in(0,\pi)$, and the $1$-forms $r_3dt$ and $r_{10}dx_0+r_{11}dx_1$ are $h_1$-compatible iff $r_3(t)=\cos t\cdot r_{10}(\sin t,0)$ for all $t\in(0,\pi)$. We thus obtain that $\Omega^0(X_3\cup_{\bar{h}_1}(X_1\cup_fX_2))$ is diffeomorhic to the subset of $\Omega^0(X_3)\oplus\Omega^0(X_1)\oplus\Omega^0(X_2)$ consisting of triples $(r_3,r_1,r_2))$ such that $r_1(t,0)=r_1(0,t)=r_2(t)$ for all $t>0$, $r_3(t)=r_3(\pi-t)=r_1(\sin  t,0)$ for all $t\in(0,\pi)$, while  $\Omega^1(X_3\cup_{\bar{h}_1}(X_1\cup_fX_2))$ is diffeomorhic to the subset of $\Omega^1(X_3)\oplus\Omega^1(X_1)\oplus\Omega^1(X_2)$ consisting of triples $(r_3dt,r_{10}dx_0+r_{11}dx_1,r_2ds))$ such that $r_{10}(t,0)=r_{11}(0,t)=r_2(t)$ for all $t>0$, $r_3(t)=r_3(\pi-t)=r_{10}(\sin t,0)$ for all $t\in(0,\pi)$.

Let now $W_2=(0,2)$ and $h_2(t)=(t-1)^2$. Then $\Omega^0(X_3\cup_{\bar{h}_2}(X_1\cup_fX_2))$ is diffeomorhic to the subset of $\Omega^0(X_3)\oplus\Omega^0(X_1)\oplus\Omega^0(X_2)$ consisting of triples $(r_3,r_1,r_2))$ such that  $r_1(t,0)=r_1(0,t)=r_2(t)$ for all $t>0$ and $r_3(t)=r_2((t-1)^2)$ for all $t\in(0,2)$, while $\Omega^1(X_3\cup_{\bar{h}_1}(X_1\cup_fX_2))$ is diffeomorhic to the subset of $\Omega^1(X_3)\oplus\Omega^1(X_1)\oplus\Omega^1(X_2)$ consisting of triples $(r_3dt,r_{10}dx_0+r_{11}dx_1,r_2ds))$ such that $r_{10}(t,0)=r_{11}(0,t)=r_2(t)$ for all $t>0$ and $r_3(t)=2(t-1)r_2((t-1)^2)$ for all $t\in(0,2)$.
\end{example}

We can also establish an analogue of the inheritance of the $\mathcal{D}_f^{i^*}(Y)=\mathcal{D}^{f^*j^*}(Y)$ condition (an analogue of Proposition~\ref{inherit:pshfrd:eq:prop}), which is as follows.

\begin{prop}
Let $f,h_1$ be such that $\mathcal{D}_f^{i^*}(Y)=\mathcal{D}^{f^*j^*}(Y)$ and $\mathcal{D}_{h_1,f}^{k_1^*}(W_1)=\mathcal{D}_f^{h_1^*l_1^*}(W_1)$, where $\mathcal{D}_{h_1,f}^{k_1^*}(W_1)$ is the pushforward of the diffeology on $\Omega_{h_1,f}^m(X_3)$ by the map $k_1^*$, and $\mathcal{D}_f^{h_1^*l_1^*}(W_1)$ is the pushforward of the diffeology on $\Omega_f^m(X_1)$ by the map $h_1^*l_1^*$. Then $\mathcal{D}_{\bar{h}_1}^{k_1^*}(W_1)=\mathcal{D}^{\bar{h}_1^*\bar{l}_1^*}(W_1)$, where $\mathcal{D}_{\bar{h}_1}^{k_1^*}(W_1)$ is the pushforward of the diffeology on $\Omega_{\bar{h}_1}^m(X_3)$ by $k_1^*$ and $\mathcal{D}^{\bar{h}_1^*\bar{l}_1^*}(W_1)$ is the pushforward of the diffeology on $\Omega^m(X_1\cup_fX_2)$ by the map $\bar{h}_1^*\bar{l}_1^*$. 

Likewise, let $f,h_2$ be such that $\mathcal{D}_f^{i^*}(Y)=\mathcal{D}^{f^*j^*}(Y)$ and $\mathcal{D}_{h_2}^{k_2^*}(W_1)=\mathcal{D}^{h_2^*l_2^*}(W_2)$. Then $\mathcal{D}_{\bar{h}_2}^{k_2^*}(W_2)=\mathcal{D}^{\bar{h}_2^*\bar{l}_2^*}(W_2)$ with notation similarly defined.
\end{prop}

\begin{proof}
Let $q_3$ be a plot of $\Omega_{h_1,f}^m(X_3)=\Omega_{\bar{h}_1}^m(X_3)$. Then locally there exist $q_1$ a plot of $\Omega_f^m(X_1)$ such that $k_1^*q_3=h_1^*l_1^*q_1$ (locally) and $q_2$ a plot of $\Omega^m(X_2)$ such that $i^*q_1=f^*j^*q_2$. Thus, up to appropriately shrinking the domain of definition of $q_3$, we have $k_1^*q_3=\bar{h}_1^*\bar{l}_1^*(q_1\cup_fq_2)$, where $q_1\cup_fq_2$ is a plot of $\Omega^m(X_1\cup_fX_2)$ defined pointwise by $(q_1\cup_fq_2)(u)=q_1(u)\cup_fq_2(u)$.

\emph{Vice versa}, let $q_{1,2}$ be a plot of $\Omega^m(X_1\cup_fX_2)$. Then locally $q=q_1\cup_fq_2$, where $q_1$ is a plot of $\Omega_f^m(X_1)$. Therefore the locally exists a plot $q_3$ of $\Omega_{h_1,f}^m(X_3)$ such that $k_1^*q_3=h_1^*l_1^*q_1=\bar{h}_1^*\bar{l}_1^*q_{1,3}$.

The case of $h_2$ is treated analogously.
\end{proof}

\begin{example}
Let us consider the case of $h_1$ and $m=0$ (the other cases are treated analogously). The condition $\mathcal{D}_f^{i^*}(Y)=\mathcal{D}^{f^*j^*}(Y)$ then corresponds to there being, for every smooth map $P_1:U\times\mathbb{R}^2\to\mathbb{R}$ such that $P_1(u,t,0)=P_1(u,0,t)\,\,\forall u\,\forall t>0$, a smooth map $P_2:U\times\mathbb{R}\to\mathbb{R}$ such that $P_1(u,0,t)=P_2(u,t)\,\,\forall t>0$, and \emph{vice versa} (this is trivially satisfied; given $P_1$, we set $P_2(u,t)=P_1(u,0,t)$ for all $u,t$, and given $P_2$ we set $P_1(u,x_0,x_1)=P_2(u,x_0+x_1)$ for all $u,x_0,x_1$). The condition $\mathcal{D}_{h_1,f}^{k_1^*}(W_1)=\mathcal{D}_f^{h_1^*l_1^*}(W_1)$ corresponds to there being, for every smooth map $P_1:U\times\mathbb{R}^2\to\mathbb{R}$ as above, a smooth map $P_3:U\times\mathbb{R}\to\mathbb{R}$ such that $P_3(u,t)=P_3(u,\pi-t)=P_1(u,\sin t,0)\,\,\forall u\,\forall t\in(0,\pi)$ (given $P_1$, $P_3$ is defined in the obvious way, while, at least for constant plots, the reverse appears to follow from the argument outlined in the footnote to Example~\ref{phfrd:dfgies:expl}).
\end{example}

\vspace{1cm}

\noindent University of Pisa \\
Department of Mathematics \\
Via F. Buonarroti 1C\\
56127 Pisa, Italy\\
\ \\
ekaterina.pervova@unipi.it\\


\begin{thebibliography}{99}
\bibitem{BH}
J. Baez, A. Hoffnung. \textit{Convenient categories of smooth spaces}. Trans. Amer. Math. Soc. 363 (2011), 5789-5825.
\bibitem{extensions}
A. Brudnyi, Yu. Brudnyi. \textit{Methods of Geometric Analysis in Extension and Trace Problems}. Birkh\"auser MMA Vol. 102 (2012).
\bibitem{chen1}
K.T. Chen. \textit{Iterated integrals of differential forms and loop space homology}. Ann. of Math.(2) 97 (1973), 217-246.
\bibitem{chen2}
K.T. Chen. \textit{Iterated integrals, fundamental groups and covering spaces}. Trans. Amer. Math. Soc. 206 (1975), 83-98.
\bibitem{chen3}
K.T. Chen. \textit{Iterated path integrals}. Bull. Amer. Math. Soc. 83 (1977), 831-879.
\bibitem{CWtriang}
J.D. Christensen, E. Wu. \textit{The homotopy theory of diffeological spaces}. New York J. Math. 29 (2014), 1269-1303.
\bibitem{CSW_Dtopology}
J.D. Christensen, G. Sinnamon, E. Wu. \textit{The D-topology for diffeological spaces}. Pacific J. Math (1) 272 (2014), 87-110.
\bibitem{CWtangent}
J.D. Christensen, E. Wu.  \textit{Tangent spaces and tangent bundles for diffeological spaces}. Cahiers de Topologie et G\'eom\'etrie Differentielle LVII (2016), 3-50.
\bibitem{CWextbund}
J.D. Christensen, E. Wu. \textit{Exterior bundles in diffeology}. Israel J. Math. 253 (2023), 673-713.
\bibitem{frol}
A. Fr\"{o}licher. \textit{Smooth structure}. In: Lecture notes in math. Vol. 962. Berlin 1982. Springer, 69-81.
\bibitem{optim1}
N. Goldammer, K. Welker. \textit{Towards optimization techniques on diffeological spaces by generalizing Riemannian concepts}. arXiv:2009.04262.
\bibitem{optim2}
N. Goldammer, K. Welker. \textit{Optimization on diffeological spaces}. Proceedings in applied mathematics and mechanics. Vol. 21. Issue S1. Special issue: 8th SAMM Juniors' Summer School (SAAM 2021). Article e202100260.
\bibitem{optim3}
N. Goldammer, J.-P. Magnot, K. Welker. \textit{On diffeologies from infinite-dimensional geometry to PDE constrained optimization}. Contemp. Math. AMS. 794 (2024), 1-48.
\bibitem{pol-surf}
F. G\"{u}nter, C. Jiang, H. Pottman. \textit{Smooth polyhedral surfaces}. Advances in Mathematics 25 (2020), 107004.
\bibitem{GI-Z18}
S. G\"{u}rer, P. Iglesias-Zemmour. \textit{Differential forms on stratified spaces II}. Bull. Australian Math. Soc. (2018) 98 (2), 319-330.
\bibitem{GI-Z19}
S. G\"{u}rer, P. Iglesias-Zemmour. \textit{Differential forms on manifolds with boundary and corners}. Indagationes Mathematicae (2019) 30 (5), 920-929.
\bibitem{hector}
G. Hector. \textit{G\'eom\'etrie et topologie des espaces diff\'eologiques}. Amalysis and geometry in foliated manifolds (Santiago de Compostela 1994), 55-80. World Sci. Publ., River Edge NJ, 1995.
\bibitem{iglFibre}
P. Iglesias-Zemmour. \textit{Fibrations diff\'eologiques et homotopie}, Th\`ese de doctorat d'\'Etat, Universit\'e de Provence, Marseille, 1985.
\bibitem{iglesiasBook}
P. Iglesias-Zemmour. \textit{Diffeology}. Mathematical Surveys and Monographs, 185, AMS, Providence, 2013.
\bibitem{I-Z24}
P. Iglesias-Zemmour. \textit{Lecture notes  on diffeology}. To be published.
\bibitem{II2015}
N. Iwase, N. Izumida. \textit{Mayer-Vietoris sequence for differentiable/diffeological spaces}. In: Algebraic Topology and Related Topica (Mohali 2017), Trends in Mathematics, Birkhauser 2019.
\bibitem{karshon-watts}
Y. Karshon, J. Watts. \textit{Basic forms and orbit spaces: a diffeological approach}, SIGMA, Symmetry Integrability Geom. Methods Appl. 12 Paper 026, 19 p., 2016.
\bibitem{KK2020}
K. Kuribayashi. \textit{Simplicial cochain algebra for diffeological spaces}. Indag. Math. (N.S.) (6) 31 (2020),
 934-967.
\bibitem{KK2021}
K. Kuribayashi. \textit{A comparison between two de Rham complexes in diffeology}. Proc. 
Amer. Math. Soc. (11) 149 (2021), 4963-4972.
\bibitem{magnotTriang1}
J.-P. Magnot. \textit{Differentiation on spaces of triangulations and optimized triangulations}. J. Phys.: Conf. Ser. (1) 738 articleID 012088 (2016).
\bibitem{magnotTriang2}
J.-P. Magnot. \textit{On the differential geometry of mumerical schemes and weak solutions of functional equations}. Nonlinearity (12) 33 (2020), 6835-6867.
\bibitem{ntumba}
P. Ntumba. \textit{DW complexes and their underlying topological spaces}. Quaestiones Math. 25 (2002), 119-134.
\bibitem{pseudobundles}
E. Pervova. \textit{Diffeological vector pseudo-bundles}. Topology and Its Applications \textbf{202} (2016), 269-300.
\bibitem{formsII}
E. Pervova. \textit{Differential forms on diffeological spaces and diffeological gluing, II}, in preparation.
\bibitem{forms-smplx}
E. Pervova. \textit{Differential forms on diffeological simplexes}, in preparation.
\bibitem{polya}
G. Polya. \textit{An elementary analogue of the Gauss-Bonnet theorem}. The Amer. Math. Monthly (9) 61 (1954), 601-603.
\bibitem{sasin91}
W. Sasin. \textit{Geometrical properties of gluing of differential spaces}. Demonstratio Mathematica (1991) XXIV (3-4), 635-656.
\bibitem{sasin92}
W. Sasin. \textit{Gluing of differential spaces}. Demonstratio Mathematica (1992) XXV (1-2), 361-384.
\bibitem{seeley}
R.T. Seeley. \textit{Extension of $C^\infty$ functions defined in a half space}. Proc. AMS (4) 15 (1964), 625-626.
\bibitem{sikorski}
R. Sikorski. \textit{Differential modules}. Colloq. Math. 24 (1971), 45-79.
\bibitem{sikorsi2}
R. Sikorski. \textit{Abstract covariant derivative}. Colloq. Math. 18 (1967), 251-272.
\bibitem{rafts}
K. Simons, E. Ikonen. \textit{Functional rafts in cell membranes}. Nature 387 (1997), 569-572.
\bibitem{smith}
J.W. Smith. \textit{The de Rham theorem for general spaces}. T$\hat{o}$hoku Math. J. (2) 18 (1966), 115-137.
\bibitem{So1}
J.M. Souriau. \textit{Groups diff\'erentiels}. Differential geometrical methods in mathematical physics (Proc. Conf., Aix-en-Provence/Salamanca, 1979), Lecture Notes in Mathematics 836,
Springer (1980), 91-128.
\bibitem{So2}
J.M. Souriau. \textit{Groups diff\'erentiels de physique math\'ematique}. South Rhone seminar on geometry, II (Lyon, 1984), Ast\'erisque 1985, Num\'ero Hors S\'erie, 341-399.
\bibitem{St}
A. Stacey. \textit{Comparative smootheology}. Theory Appl. Categ., (4) 25 (2011), 64-117.
\bibitem{vincent}
M. Vincent. \textit{Diffeological differential geometry}. Master Thesis, University of Copenhagen, 2008.
\bibitem{watts}
J. Watts. \textit{Diffeologies, Differential Spaces, and Symplectic Geometry}. PhD Thesis, 2012, University of Toronto, Canada.
\bibitem{ww}
J. Watts, S. Wolbert. \textit{Diffeological coarse moduli spaces of stacks over manifolds}. Contemp. Math. AMS 794 (2024), 161-178.
\bibitem{whit43}
H. Whitney. \textit{Differentiable even functions}. Duke Math. J. 10 (1943), 159-160.
\bibitem{wu}
E. Wu. \textit{Homological algebra for diffeological vector spaces}. Homology Homotopy Appl. (1) 17 (2015), 339-376.
\bibitem{wu24}
E. Wu. \textit{A survey on diffeological vector spaces and applications}. Contemp. Math. AMS 794 (2024), 113-128.
\bibitem{rafts2}
J. Zimmerberg, M.M. Kozlov. \textit{How proteins produce cellular membrane curvature}. Nat. Rev. Mol. Cell Biol. 7 (2006), 9-19.
\end{thebibliography}
\end{document}